\documentclass[14pt,leqno]{article}

\usepackage{amsmath,amssymb,amsthm}

\usepackage[margin=3cm]{geometry} 

\usepackage{titlesec,hyperref}

\usepackage{color}

\usepackage{fancyhdr}
\titleformat{\subsection}{\it}{\thesubsection.\enspace}{1pt}{}

\newtheorem{theo}{Theorem}[section]
\newtheorem{lemm}[theo]{Lemma}

\numberwithin{equation}{section}

\allowdisplaybreaks 

\begin{document}
\title{Global weak solutions for a generalized Camassa-Holm equation
\hspace{-4mm}
}

\author{Xi $\mbox{Tu}^1$\footnote{E-mail: tuxi@mail2.sysu.edu.cn} \quad and\quad
 Zhaoyang $\mbox{Yin}^{1,2}$\footnote{E-mail: mcsyzy@mail.sysu.edu.cn}\\
 $^1\mbox{Department}$ of Mathematics,
Sun Yat-sen University,\\ Guangzhou, 510275, China\\
$^2\mbox{Faculty}$ of Information Technology,\\ Macau University of Science and Technology, Macau, China}

\date{}
\maketitle
\hrule

\begin{abstract}
In this paper we mainly investigate the Cauchy problem of a generalized Camassa-Holm equation. First by the relationship between the Degasperis-Procesi equation and the generalized Camassa-Holm equation, we obtain two global existence results and two blow-up results. Then, we prove the existence and uniqueness of global weak solutions under some certain sign condition.\\

\vspace*{5pt}

\vspace*{5pt}
\end{abstract}

\vspace*{10pt}

\tableofcontents

%
%
%


\section{Introduction}
In this paper we consider the Cauchy problem for the following generalized Camassa-Holm equation,
\begin{align}\label{E1}
\left\{
\begin{array}{ll}
u_t-u_{txx}=\partial_x(2+\partial_x)[(2-\partial_x)u]^2,~~~~  t>0,\\[1ex]
u(0,x)=u_{0}(x).
\end{array}
\right.
\end{align}
Note that $G(x)=\frac{1}{2}e^{-|x|}$ and $G(x)\star f =(1-\partial _x^2)^{-1}f$ for all $f \in L^2(\mathbb{R})$ and $G \star m=u$. Then we can rewrite (\ref{E1}) as follows:
\begin{align}\label{E2}
\left\{
\begin{array}{ll}
u_t-4uu_x=-u^2_{x}+G\ast[\partial_x(2u_x^2+6 u^{2})+u^2_{x}], ~~ t>0,\\[1ex]
u(0,x)=u_{0}(x).
\end{array}
\right.
\end{align}
The equation (\ref{E1}) was proposed recently by Novikov in \cite{n1}. He showed that the equation (1.1) is integrable by using as definition of integrability the existence of an infinite hierarchy of quasi-local higher symmetries \cite{n1} and it belongs to the following class \cite{n1}:
\begin{align}\label{E02}
(1-\partial^2_x)u_t=F(u,u_x,u_{xx},u_{xxx}),
\end{align}
which has attracted much interest, particularly in the possible integrable members of (\ref{E02}).

The most celebrated integrable member of (\ref{E02}) is the well-known Camassa-Holm (CH) equation \cite{Camassa}:
\begin{align}
(1-\partial^2_x)u_t=3uu_x-2u_{x}u_{xx}-uu_{xxx}.
\end{align}
The CH equation can be regarded as a shallow water wave equation \cite{Camassa, Constantin.Lannes}.  It is completely integrable. That means that the system can be transformed into a linear flow at constant speed in suitable action-angle variables (in the sense of infinite-dimensional Hamiltonian systems), for a large class of initial data \cite{Camassa,Constantin-P,Constantin.mckean}.
 It also has a bi-Hamiltonian structure \cite{Constantin-E,Fokas}, and admits exact peaked solitons of the form $ce^{-|x-ct|}$ with $c>0$, which are orbitally stable \cite{Constantin.Strauss}. It is worth mentioning that the peaked solitons present the characteristic for the traveling water waves of greatest height and largest amplitude and arise as solutions to the free-boundary problem for incompressible Euler equations over a flat bed, cf. \cite{Camassa.Hyman,Constantin2,Constantin.Escher4,Constantin.Escher5,Toland}.

The local well-posedness for the Cauchy problem of the CH equation in Sobolev spaces and Besov spaces was discussed in \cite{Constantin.Escher,Constantin.Escher2,d1,d2,Guillermo}. It was shown that there exist global strong solutions to the CH equation \cite{Constantin,Constantin.Escher,Constantin.Escher2} and finite time blow-up strong solutions to the CH equation \cite{Constantin,Constantin.Escher,Constantin.Escher2,Constantin.Escher3}. The existence and uniqueness of global weak solutions to the CH equation were proved in \cite{Constantin.Molinet, Xin.Z.P}. The global conservative and dissipative solutions of CH equation were investigated in \cite{Bressan.Constantin,Bressan.Constantin2}.

The second celebrated integrable member of (\ref{E02}) is the famous Degasperis-Procesi (DP) equation \cite{D-P}:
\begin{align}
(1-\partial^2_x)u_t=4uu_x-3u_{x}u_{xx}-uu_{xxx}.
\end{align}
The DP
equation can be regarded as a model for nonlinear shallow water
dynamics and its asymptotic accuracy is the same as for the
CH shallow water equation \cite{D-G-H}. The DP equation is integrable and has a bi-Hamiltonian structure \cite{D-H-H}. An inverse scattering approach for the
DP equation was presented in \cite{Constantin.lvanov.lenells,Lu-S}. Its
traveling wave solutions was investigated in \cite{Le,V-P}. \par
The
local well-posedness of the Cauchy problem of the DP equation in Sobolev spaces and Besov spaces was established in
\cite{G-L,H-H,y1}. Similar to the CH equation, the
DP equation has also global strong solutions
\cite{L-Y1,y2,y4} and finite time blow-up solutions
\cite{E-L-Y1, E-L-Y,L-Y1,L-Y2,y1,y2,y3,y4}. It also has global weak
solutions \cite{C-K,E-L-Y1,y3,y4}.
\par
Although the DP equation is similar to the
CH equation in several aspects, these two equations are
truly different. One of the novel features of the DP
different from the CH equation is that it has not only
peakon solutions \cite{D-H-H} and periodic peakon solutions
\cite{y3}, but
also shock peakons \cite{Lu} and the periodic shock waves \cite{E-L-Y}.

The third celebrated integrable member of (\ref{E02}) is the known Novikov equation \cite{n1}:
\begin{align}
(1-\partial^2_x)u_t=3uu_{x}u_{xx}+u^2u_{xxx}-4u^2u_x.
\end{align}
The most difference between the Novikov equation and the CH and DP equations is that the former one has cubic nonlinearity and the latter ones have quadratic nonlinearity.

It was showed that the Novikov equation is integrable, possesses a bi-Hamiltonian structure, and admits exact peakon solutions $u(t,x)=\pm\sqrt{c}e^{|x-ct|}$ with $c>0$ \cite{Hone}.\\
$~~~~~~$ The local well-posedness for the Novikov equation in Sobolev spaces and Besov spaces was studied in \cite{Wu.Yin2,Wu.Yin3,Wei.Yan,Wei.Yan2}. The global existence of strong solutions under some sign conditions were established in \cite{Wu.Yin2} and the blow-up phenomena of the strong solutions were shown in \cite{Wei.Yan2}. The global weak solutions for the Novikov equation were studied in \cite{Laishaoyong,Wu.Yin}.

Recently, the Cauchy problem of (\ref{E1}) in the Besov spaces $B^{s}_{p,r},~s>max\{2+\frac{1}{p},~\frac{5}{2}\},~1\leq p,q \leq \infty$ and the critical Besov space $B^{\frac{5}{2}}_{2,1}$ has been studied in \cite{Tu-Yin1,Tu-Yin2}.
To our best knowledge, the global solution of (\ref{E1}) has not been studied yet. In this paper we first investigate the global solution of (\ref{E2}) with initial data in the Sobolev space $H^{s},~s>\frac{5}{2}$ and two new blow-up results. The main idea is based on the global solution of DP equation with initial data in the Sobolev space
$H^{s},~s>\frac{3}{2}$. Then, we prove the existence and uniqueness of global weak solution by using the method of approximation of smooth solutions and a regularization technique.

The paper is organized as follows. In Section 2 we introduce some preliminaries which will be used in the sequel. In Section 3 we show global existences and blow-up phenomena. Section 4 is devoted to the study of global weak solutions of (\ref{E1}).
\vspace*{2em}
\section{Preliminaries}

\begin{theo}\label{Thm1}\cite{Tu-Yin1}
Let $1\leq p,~r\leq \infty,~s>\max\{2+\frac{1}{p},\frac{5}{2}\},$ and  $~u_0\in B^s_{p,r}.$ Then there exists some $T>0$, such that (\ref{E2}) has a unique solution $u$ in
$ E^s_{p,r}(T)$, where
\begin{align}
   E^s_{p,r}(T)\triangleq
\left\{
\begin{array}{ll}
C([0,T);B^s_{p,r})\cap C^1([0,T);B^{s-1}_{p,r}),~~~~if~r<\infty, \\[1ex]
C_w([0,T);B^s_{p,\infty})\cap C^{0,1}([0,T);B^{s-1}_{p,\infty}),~~~~if~r=\infty.\\[1ex]
\end{array}
\right.
\end{align}
 Moveover the solution depends continuously on the initial data $u_{0}.$
\end{theo}
From \cite{Tu-Yin1}, for every $t>0$, we have that
\\(1) $u(t,\cdot)-u_{xx}(t,\cdot)\leq0$ and $|u_{x}(t,\cdot)|\leq u(t,\cdot)$ on $\mathbb{R}$.
\\(2) $\|u(t,\cdot)\|_{L^1(\mathbb{R})}=\|m(t,\cdot)\|_{L^{1}(\mathbb{R})}=\|m_{0}\|_{L^1(\mathbb{R})}$.
\\(3) $\|u_{x}(t,\cdot)\|_{L^{\infty}(\mathbb{R})}\leq \|m_{0}\|_{L^1(\mathbb{R})}$ and
$\|u(t,\cdot)\|_{H^1(\mathbb{R})}=\|u_{0}\|_{H^1(\mathbb{R})}$.

\begin{lemm} \label{lem1}\cite{malek1}
Let $T >0$. If
$$f,g \in L^2((0, T); H^1(\mathbb{R}))~~~\text{ and}~~~
\frac{df}{dt},~\frac{dg}{dt}
\in L^2((0, T);H^1(\mathbb{R})),$$
then $f, g$ are a.e. equal to a function continuous from [0,T] into $L^2(\mathbb{R})$ and
$$<f (t); g(t)> - <f(s); g(s)>=
\int^{t}_{s}<\frac{df(\tau)}{d\tau},g(\tau)>d\tau+
\int^{t}_{s}<\frac{dg(\tau)}{d\tau},f(\tau)>d\tau
$$
for all $s,t\in[0, T]$.
\end{lemm}

In this paper, we will define $\{\rho_{n}\}_{n\geq1}$ by the mollifiers
$$\rho_{n}(x):=n\rho(nx)\bigg(\int_{\mathbb{R}}\rho(\xi)d\xi\bigg)^{-1},~~x\in\mathbb{R},~ n\geq1,
$$
where $\rho \in C_{c}^{\infty}$ is denoted by
\begin{align}\label{l41}
\rho(x)=\left \{\begin {array}{ll}e^{\frac{1}{x^2-1}},
\ \ \ \ \ \ \ &|x|\leq 1,\\0, \ \ \ \ \ \ & |x|\geq 1.
\end {array}\right.
\end{align}
Then, we recall several useful approximation results.

\begin{lemm}\label{lem2}\cite{Constantin.Molinet}
Let $f : \mathbb{R} \rightarrow \mathbb{R}$ be uniformly continuous and bounded. If $\mu \in \mathcal{M}(\mathbb{R})$, then
$$[\rho_{n}\ast(f\mu)-(\rho_{n}\ast(f))(\rho_{n}\ast(\mu))]\rightarrow 0 ~~\text{in}~~ L^1(\mathbb{R})~ ~\text{as}~ ~n \rightarrow \infty.$$
\end{lemm}

\begin{lemm}\label{lem3}\cite{Constantin.Molinet}
Let $f : \mathbb{R} \rightarrow \mathbb{R}$ be uniformly continuous and bounded. If $g\in L^{\infty}(\mathbb{R})$, then
$$[\rho_{n}\ast(f\mu)-(\rho_{n}\ast(f))(\rho_{n}\ast(\mu))]\rightarrow 0 ~~\text{in}~~ L^\infty(\mathbb{R})~ ~\text{as}~ ~n \rightarrow \infty.$$
\end{lemm}

\begin{lemm}\label{lem9}\cite{lions1,gy1}
Let $f \in W^{1,p}(\mathbb{R})$ and $g\in L^{q}(\mathbb{R})$, with $1\leq q\leq \infty$,
then
$$[\rho_{n}\ast(fg')-f(\rho_{n}\ast g)']\rightarrow 0 ~~\text{in}~~ L^r(\mathbb{R}),~~\text{where} ~~\frac{1}{r}=\frac{1}{p}+\frac{1}{q}.$$
\end{lemm}

\begin{lemm} \cite{Natanson}Helly's theorem

Assume \,$F$\, is a family of real functions defined on \,$[a,b]$. For any \,$f\in F$\, satisfies \,$|f|\leq K$\,
and \,$V(f)\leq K$\, with the some constant \,$K\geq0$.
Then there exists a family of functions \,$\{f_{n}\}_{n=1}^{\infty}\subset F$\, such that
$$\lim_{n\rightarrow\infty}f_{n}(t)=g(t) ~~and~~g(t)\in BV[a,b],
$$
for any \,$t\in[a,b]$.
\end{lemm}

\section{Global existence and blow-up phenomena}

In this section, we begin by deriving the relationship between  the Degasperis-Procesi equation and (\ref{E1}). Using this relationship, we then obtain the global existence results and the blow-up results for the system (\ref{E1}).

We can rewrite the Degasperis-Procesi equation as
\begin{align}
(1-\partial_{x^2})(\frac{1}{2}\widetilde{u}_{t})=\partial_{x}(4-\partial_{x}^2)(\frac{1}{2}\widetilde{u})^2.
\end{align}
It is easy to see that the Degasperis-Procesi equation transforms into the  equation (\ref{E1}) under the transformation
$$\widetilde{u}\rightarrow 2(2-\partial_{x})u.$$

Applying the above relation and the theorems in \cite{liu-yin1}, we can obtain that

\begin{theo}\label{thm}
Assume $u_{0}\in H^s (\mathbb{R})$, $s > \frac{5}{2}$. If $\widetilde{m_{0}} =2(2-\partial_{x})m_{0}=2(2-\partial_{x})(u_{0}-u_{0xx})$ does not change sign on $\mathbb{R}$, then (\ref{E1}) has a global strong solution $u = u(\cdot, u_0) \in C([0,\infty); H^s (\mathbb{R})) \cap C^1([0,\infty); H^{(s-1)}(\mathbb{R}))$.
\end{theo}

\begin{theo}
Assume $u_0
\in H^s (\mathbb{R})$, $s > \frac{5}{2}$. Then there exists $x_{0} \in \mathbb{R}$ such that
\begin{align}
\nonumber &\widetilde{m_{0}}(x)
=2(2-\partial_{x})m_{0}=2(2-\partial_{x})(u_{0}-u_{0xx}) \leq 0 ~~~~~~\text{if}~~ x \leq x_{0},
\\\nonumber
&\widetilde{m_{0}}(x)
=2(2-\partial_{x})m_{0}=2(2-\partial_{x})(u_{0}-u_{0xx}) \geq 0~~~~~~ \text{if}~~ x \geq x_{0}.
\end{align}
Then (\ref{E1}) has a unique global strong solution
$u = u(\cdot, u_{0}) \in C([0,\infty); H^s (\mathbb{R})) \cap C^1([0,\infty); H^{s-1}(\mathbb{R}))$.

\end{theo}

\begin{theo}
 Let $u_{0}\in H^s (\mathbb{R}),~ s > \frac{5}{2}$. Assume there exists $x_{0} \in \mathbb{R}$ such that
\begin{align}
\nonumber &\widetilde{m_{0}}(x)
=2(2-\partial_{x})m_{0}=2(2-\partial_{x})(u_{0}-u_{0xx})  \geq 0 ~~~~~~\text{if}~~ x \leq x_{0},\\\nonumber
&\widetilde{m_{0}}(x)
=2(2-\partial_{x})m_{0}=2(2-\partial_{x})(u_{0}-u_{0xx})  \leq 0~~~~~~ \text{if}~~ x \geq x_{0}.
\end{align}
and $m_{0}$ changes sign. Then, the corresponding solution to (\ref{E1}) blows up in finite time.
\end{theo}

\begin{theo}
Let $\varepsilon > 0$, $u_{0} \in H^s (\mathbb{R}), ~s > \frac{5}{2}$. Assume there is $x_{0} \in \mathbb{R}$ such that
$$\widetilde{u}'_{0}(x_{0})=2\partial_{x}(2-\partial_{x})u_{0}(x_{0}) < -\frac{(1 + \varepsilon)\sqrt{6}}{4}\bigg(\|\widetilde{u}_{0}\|_{L^\infty}
+(2\sqrt{6}\|\widetilde{u}_{0}\|^2_{L^2}\ln(1+\frac{2}{\varepsilon})
+\|\widetilde{u}_{0}\|^2_{L^\infty})^{\frac{1}{2}}
\bigg).$$
Then the corresponding solution to (\ref{E1}) blows up in finite time. Moreover, the
maximal time of existence is estimated above by
$$\frac{(2\sqrt{6}\|\widetilde{u}_{0}\|^2_{L^2}\ln(1+\frac{2}{\varepsilon})
+\|\widetilde{u}_{0}\|^2_{L^\infty})^{\frac{1}{2}}
-\|\widetilde{u}_{0}\|_{L^\infty}}{6\|\widetilde{u}_{0}\|^2_{L^2}}.$$
\end{theo}

\section{Global weak solutions}
In this section, we will investigate the existence and uniqueness of global weak solutions of (\ref{E1}) under some sign condition.

\begin{theo}\label{THM1}
If $u_{0}\in H^{1}(\mathbb{R})$ is such that $m_{0}=u_{0}-u_{0xx}\in \{\mathcal{M}^{}(\mathbb{R})|(2-\partial_{x})m_{0}\in\mathcal{M}^{+}(\mathbb{R})\}$. Then (\ref{E1}) has a unique solution $u\in C^1(\mathbb{R}_{+}; L^{2}(\mathbb{R}))\cap C(\mathbb{R}_{+}; H^{1}(\mathbb{R}))\cap C_{w}(\mathbb{R}_{+}; H^{2}(\mathbb{R}))$ with initial data $u(0)=u_{0}$.
Moveover the total variation of $(2-\partial_{x})m(t,\cdot)=(2-\partial_{x})[u(t,\cdot)-u_{xx}(t,\cdot)]\in \mathcal{M}^{+}(\mathbb{R})$ is uniformly bounded on $\mathbb{R}_{+}$ and $E(u)=\int_{\mathbb{R}}(u^2+u_{x}^2)dx$ is conservation law.
\end{theo}

\begin{proof}
In order to prove Theorem \ref{THM1}, we proceed as the following five steps.\\
 {\bf Step 1:} We define a sequence of global solutions $u^{n}\in C([0,T]; H^3(\mathbb{R}))$ of (\ref{E1}) with a suitable approximation of the initial data $u_{0}\in  H^1(\mathbb{R})$ by smooth functions $u_{0}^{n}$.

Let $u_{0}\in H^{1}(\mathbb{R})$, $m_{0}=u_{0}-u_{0xx}\in \mathcal{M}(\mathbb{R}) $ and $(2-\partial_{x})m_{0}=(2-\partial_{x})u_{0}-(2-\partial_{x})u_{0xx}\in \mathcal{M}^{+}(\mathbb{R}).$

Define $u_{0}^{n}:=\rho_{n}\ast u_{0}\in H^{\infty}(\mathbb{R})$ for $n\leq1$. Then we have
\begin{align}\label{l10}
u_{0}^{n}\rightarrow u_{0}~~~\text{in} ~~H^{1}(\mathbb{R})~~~\text{as}~~n\rightarrow \infty.
\end{align}
Note that for $n\geq1$,
\begin{align}\label{l1}
\|(2-\partial_{x})u^{n}\|_{L^1(\mathbb{R})}&
=\|(2-\partial_{x})m^{n}\|_{L^1(\mathbb{R})}=\|(2-\partial_{x})m_{0}^{n}\|_{L^1(\mathbb{R})}
\leq \|2m_{0}-m_{0}'\|_{\mathcal{M}},
\\
\|u^{n}\|_{H^1(\mathbb{R})}&=\|u_{0}^{n}\|_{H^1(\mathbb{R})}\leq\|u_{0}\|_{H^1(\mathbb{R})},
\\\nonumber
\|(2-\partial_{x})u^{n}\|_{H^1(\mathbb{R})}&=\|G\ast (2-\partial_{x})m^{n}\|_{L^2(\mathbb{R})}
+\|G_{x}\ast (2-\partial_{x})m^{n}\|_{L^2(\mathbb{R})}
\\\nonumber &\leq\|G\|_{L^2(\mathbb{R})}\|(2-\partial_{x})m^{n}\|_{L^1(\mathbb{R})}
+\|G_{x}\|_{L^2(\mathbb{R})}\|(2-\partial_{x})m^{n}\|_{L^1(\mathbb{R})}
\\\nonumber &\leq2\|(2-\partial_{x})m_{0}^{n}\|_{L^1(\mathbb{R})}
\leq2\|2m_{0}-m_{0}'\|_{\mathcal{M}},
\\
(2-\partial_{x})m_{0}^{n}&=(2-\partial_{x})u_{0}^{n}-(2-\partial_{x})u_{0xx}^{n}
=\rho_{n}\ast(2-\partial_{x})m_{0}\geq0.
\end{align}
Then, we derive
\begin{align}\label{l27}
\nonumber \|u^{n}\|_{H^2(\mathbb{R})}=&\|u^{n}\|_{H^1(\mathbb{R})}+\|u_{x}^{n}\|_{H^1(\mathbb{R})}
\\\nonumber\leq&3\|u^{n}\|_{H^1(\mathbb{R})}+\|(2-\partial_{x})u^{n}\|_{H^1(\mathbb{R})}
\\ \leq &3\|u_{0}\|_{H^1(\mathbb{R})}+2\|2m_{0}-m_{0}'\|_{\mathcal{M}},
\\\nonumber\|\partial_{t}u^{n}\|_{L^2((0,T)\times\mathbb{R})}
=&\|4u^{n}u^{n}_x\|_{L^2((0,T)\times\mathbb{R})}+\|(u^{n}_{x})^2\|_{L^2((0,T)\times\mathbb{R})}
\\\nonumber&+\|G\ast\{\partial_x[(2u^{n}_x)^2+(6 u^{n})^{2}]
+(u^{n}_{x})^2\}\|_{L^2((0,T)\times\mathbb{R})}
\\\nonumber\leq &4\|u^{n}\|_{L^\infty((0,T)\times\mathbb{R})}\|u^{n}_x\|_{L^2((0,T)\times\mathbb{R})}
+6\|u^{n}\|^2_{L^2((0,T)\times\mathbb{R})}
\\\nonumber&+3\|u^{n}_{x}\|^2_{L^2((0,T)\times\mathbb{R})}
+\|u^{n}_{x}\|_{L^\infty((0,T)\times\mathbb{R})}\|u^{n}_{x}\|_{L^2((0,T)\times\mathbb{R})}
\\\nonumber\leq &4\|u^{n}\|_{L^\infty((0,T)\times\mathbb{R})}\|u^{n}\|_{L^2((0,T);H^1(\mathbb{R}))}
+6\|u^{n}\|^2_{L^2((0,T)\times\mathbb{R})}
\\\nonumber&+3\|u^{n}\|^2_{L^2((0,T);H^1(\mathbb{R}))}
+\|u^{n}\|_{L^\infty((0,T);H^2(\mathbb{R}))}\|u^{n}\|_{L^2((0,T);H^1(\mathbb{R}))}
\\\leq & C_{T} (16\|u_{0}\|^2_{H^1(\mathbb{R})}+2\|u_{0}\|_{H^1(\mathbb{R})}
\|2m_{0}-m_{0}'\|_{\mathcal{M}(\mathbb{R})}),
\end{align}
and
\begin{align}\label{029}
\nonumber\|\partial_{t}u_{x}^{n}\|_{L^2((0,T)\times\mathbb{R})}
=&\|(4u^{n}-2u_{x}^{n})u^{n}_{xx}\|_{L^2((0,T)\times\mathbb{R})}
+2\|(u^{n}_{x})^2\|_{L^2((0,T)\times\mathbb{R})}
+6\|(u^{n})^2\|_{L^2((0,T)\times\mathbb{R})}
\\\nonumber&+\|G\ast\{(2u^{n}_x)^2+(6 u^{n})^{2}
+\partial_x[(u^{n}_{x})^2]\}\|_{L^2((0,T)\times\mathbb{R})}
\\\nonumber\leq &\|4u^{n}-2u_{x}^{n}\|_{L^\infty((0,T)\times\mathbb{R})}
\|u^{n}_{xx}\|_{L^2((0,T)\times\mathbb{R})}
+6\|u^{n}\|^2_{L^2((0,T)\times\mathbb{R})}
\\\nonumber&+3\|u^{n}_{x}\|^2_{L^2((0,T)\times\mathbb{R})}
+2\|u^{n}_{x}\|_{L^\infty((0,T)\times\mathbb{R})}\|u^{n}_{x}\|_{L^2((0,T)\times\mathbb{R})}
\\\nonumber&
+6\|u^{n}\|_{L^2((0,T)\times\mathbb{R})}\|u^{n}\|_{L^\infty((0,T)\times\mathbb{R})}
\\\nonumber\leq &\|4u^{n}-2u_{x}^{n}\|_{L^\infty((0,T)\times\mathbb{R})}
\|u^{n}\|_{L^2((0,T);H^2(\mathbb{R}))}
+6\|u^{n}\|^2_{L^2((0,T)\times\mathbb{R})}
\\\nonumber&+3\|u^{n}\|^2_{L^2((0,T);H^1(\mathbb{R}))}
+2\|u^{n}\|_{L^\infty((0,T);H^2(\mathbb{R}))}\|u^{n}\|_{L^2((0,T);H^1(\mathbb{R}))}
\\\nonumber&
+6\|u^{n}\|_{L^2((0,T)\times\mathbb{R})}\|u^{n}\|_{L^\infty((0,T)\times\mathbb{R})}
\\\leq & C_{T} (\|u_{0}\|_{H^1(\mathbb{R})}+\|2m_{0}-m_{0}'\|_{\mathcal{M}(\mathbb{R})})^2.
\end{align}

Let $u^{n}$ be the global solution of (\ref{E1}) guaranteed by Theorem \ref{thm} with initial data $u^n_0$.

 {\bf Step 2:} We show that the sequence defined by this procedure converges pointwise a.e. to a function $u\in H_{loc}^{1}(\mathbb{R}^{+};H_{}^{2}(\mathbb{R}))$ that satisfies (\ref{E1}) in the sense of distributions.

For fix any $T>0$, from (\ref{l27})-(\ref{029}), we verify that $\{u^{n}\}_{n\geq1}$ and $\{u_{x}^{n}\}_{n\geq1}$ is uniformly bounded in the space $H^1((0; T) \times \mathbb{R})$.
Therefore we can get a subsequence such that
\begin{align}
u^{n_{k}}\rightharpoonup u ~~\text{weakly~~ in} ~~H^1((0; T) \times \mathbb{R})~~\text{for}~~ n_{k}\rightarrow\infty,\\
u_{x}^{n_{k}}\rightharpoonup u_{x} ~~\text{weakly~~ in} ~~H^1((0; T) \times \mathbb{R})~~\text{for}~~ n_{k}\rightarrow\infty,
\end{align}
and
\begin{align}\label{l2}
u^{n_{k}} \rightarrow ~u ~~a.e.~~ \text{on} ~~(0; T) \times \mathbb{R}~~\text{as}~~ n_{k}\rightarrow\infty,
\end{align}
\begin{align}\label{l02}
u_{x}^{n_{k}} \rightarrow ~u_{x} ~~a.e.~~ \text{on} ~~(0; T) \times \mathbb{R}~~\text{as}~~ n_{k}\rightarrow\infty,
\end{align}
for $u,~u_{x}\in H^1((0; T) \times \mathbb{R})$.
For fixed $t\in(0,T)$, by (\ref{E2}), (\ref{l1})--(\ref{029}) and Young's inequality, we have
\begin{align}
\nonumber\mathbb{V} [\partial_{x}(2-\partial_{x})u^{n_{k}}]
=&\|\partial^2_{x}(2-\partial_{x})u^{n_{k}}\|_{L^{1}((0; T) \times \mathbb{R})}
+\|(2-\partial_{x})u^{n_{k}}_{tx}\|_{L^{1}((0; T) \times \mathbb{R})}
\\\nonumber \leq&\|\partial^2_{x}(2-\partial_{x})u^{n_{k}}\|_{L^{1}((0; T) \times \mathbb{R})}
+\|4u^{n_{k}}[\partial^2_{x}(2-\partial_{x})u^{n_{k}}]\|_{L^1((0; T) \times \mathbb{R})}
\\\nonumber&+\|4\partial_{x}u^{n_{k}}[\partial_{x}(2-\partial_{x})u^{n_{k}}]\|_{L^1((0; T) \times \mathbb{R})}
+\|2\partial^2_{x}u^{n_{k}}_{x}[\partial_{x}(2-\partial_{x})u^{n_{k}}]\|_{L^1((0; T) \times \mathbb{R})}
\\\nonumber&+\|2u^{n_{k}}_{x}[\partial^2_{x}(2-\partial_{x})u^{n_{k}}]\|_{L^1((0; T) \times \mathbb{R})}
+\|12u^{n_{k}}u_{x}^{n_{k}}\|_{L^1((0; T) \times \mathbb{R})}
\\\nonumber&+\|G_{x}\star[6 (u^{n_{k}})^{2}]\|_{L^1((0; T) \times \mathbb{R})}
+\|3(u^{n_{k}}_{x})^2+12(u^{n_{k}})^2\|_{L^1((0; T) \times \mathbb{R})}
\\\nonumber&+\|G\star[3(u^{n_{k}}_{x})^2+12(u^{n_{k}})^2]\|_{L^1((0; T) \times \mathbb{R})}
\\\nonumber
\leq&
\bigg[1+4\|u^{n_{k}}\|_{L^\infty((0; T) \times \mathbb{R})}+2\|u^{n_{k}}_{x}\|_{L^\infty((0; T) \times \mathbb{R})}\bigg]
\|\partial_{x}^2(2-\partial_{x})u^{n_{k}}\|_{L^{1}((0; T) \times \mathbb{R})}
\\\nonumber&+36\|u^{n_{k}}\|^2_{H^1((0; T) \times \mathbb{R})})
+\|2\partial^2_{x}u^{n_{k}}\|_{L^2((0; T) \times \mathbb{R})}
\|\partial_{x}(2-\partial_{x})u^{n_{k}}\|_{L^2((0; T) \times \mathbb{R})}
\\\nonumber&+\|4\partial_{x}u^{n_{k}}\|_{L^2((0; T) \times \mathbb{R})}\|\partial_{x}(2-\partial_{x})u^{n_{k}}\|_{L^2((0; T) \times \mathbb{R})}
\\\nonumber
\leq&\bigg[1+4\|u^{n_{k}}\|_{L^\infty((0; T) \times \mathbb{R})}+2\|u^{n_{k}}\|_{L^\infty((0; T);H^2(\mathbb{R}))}\bigg]
\\\nonumber&\times\bigg[\|(2-\partial_{x})u^{n_{k}}\|_{L^{1}((0; T) \times \mathbb{R})}+(2-\partial_{x})m^{n_{k}}\|_{L^{1}((0; T) \times \mathbb{R})}\bigg]
\\\nonumber&+2\|u^{n_{k}}\|_{L^2((0; T); H^2(\mathbb{R}))}
\|(2-\partial_{x})u^{n_{k}}\|_{L^2((0; T) ;H^1( \mathbb{R}))}
\\\nonumber&+4\|u^{n_{k}}\|_{L^2((0; T) \times \mathbb{R})}\|(2-\partial_{x})u^{n_{k}}\|_{L^2((0; T) ; H^1(\mathbb{R}))}
+36\|u^{n_{k}}\|^2_{H^1((0; T) \times \mathbb{R})}
\\\nonumber \leq&\bigg[1+10\|u_{0}\|_{H^1(\mathbb{R})}+4\|2m_{0}-m_{0}'\|_{\mathcal{M}(\mathbb{R})}\bigg]
\times2T\|2m_{0}-m_{0}'\|_{\mathcal{M}(\mathbb{R})}
\\\nonumber&+44\|u^{n_{k}}\|^2_{H^1((0; T) \times \mathbb{R})}+2\|u^{n_{k}}\|^2_{L^2((0,T);H^2(\mathbb{R}))}
+6\|u^{n_{k}}\|_{H^1((0; T) \times \mathbb{R}))}\|u^{n_{k}}\|_{L^2((0,T);H^2(\mathbb{R}))}
\\\nonumber \leq&C_{T}\bigg[1+10\|u_{0}\|_{H^1(\mathbb{R})}
+5\|2m_{0}-m_{0}'\|_{\mathcal{M}(\mathbb{R})}\bigg]^2,
\end{align}
and
\begin{align}
\|\partial_{x}(2-\partial_{x})u^{n_{k}}(t,\cdot)\|_{L^\infty(\mathbb{R})}
&\leq
\|u_{x}^{n_{k}}(t,\cdot)\|_{L^\infty(\mathbb{R})}+2\|u^{n_{k}}(t,\cdot)\|_{L^\infty(\mathbb{R})}
\\\nonumber &\leq \|u^{n}\|_{H^2(\mathbb{R})}+2\|u^{n}\|_{H^1(\mathbb{R})}
\\\nonumber &\leq5\|u_{0}\|_{H^2(\mathbb{R})}+2\|2m_{0}-m_{0}'\|_{\mathcal{M}}.
\end{align}

 Therefore the sequence $\partial_{x}(2u^{n_{k}}-u_{x}^{n_{k}})\in BV((0,T)\times\mathbb{R}),$ which is the space of functions with bounded variation and $\mathbb{V}(f)$ is the total variation of $f \in BV ((0,T)\times\mathbb{R})$.

 By Helly's theorem, $\partial_{x}(2u^{n_{k}}-u_{x}^{n_{k}})\in BV((0,T)\times\mathbb{R})$ has a subsequence, denoted again $\partial_{x}(2u^{n_{k}}-u_{x}^{n_{k}})(\cdot,\cdot)$, which converges pointwise to some function $v(\cdot,\cdot)$ where
 $\mathbb{V}[v]\leq C_{T}\bigg[1+10\|u_{0}\|_{H^1(\mathbb{R})}+5\|2m_{0}-m_{0}'\|_{\mathcal{M}(\mathbb{R})}\bigg]^2.$

From (\ref{l2}) and (\ref{l02}), we obtain
\begin{align}
\partial^2_{x}u^{n_{k}}(t, \cdot)\rightarrow \partial^2_{x}u(t, \cdot)~~\text{in } ~\mathcal{D}'(\mathbb{R}),~~\text{for~ almost~ all}~ t \in (0,T).
\end{align}

This enables that $v(\cdot,\cdot)=2u_{x}(\cdot, \cdot)-\partial_{xx}u(\cdot, \cdot)$ for a.e. $t \in (0,T).$

Therefore
\begin{align}\label{l3}
u_{xx}^{n_{k}}\rightarrow u_{xx}, ~a.e. ~\text{on}~ (0,T)\times \mathbb{R}~~~\text{as} ~~n_{k}\rightarrow\infty,
\end{align}
and
\begin{align}
\mathbb{V} [\partial_{x}(2-\partial_{x})u]=\|2u_{x}-u_{xx}\|_{\mathcal{M}}\leq
 C_{T} \bigg[1+10\|u_{0}\|_{H^1(\mathbb{R})}+5\|2m_{0}-m_{0}'\|_{\mathcal{M}(\mathbb{R})}\bigg]^2
,
 \end{align}
 for a.e. $t\in (0,T)$.

 Fix $t\in(0,T)$ again, by Theorem \ref{thm}, we obtain that the sequences $\{6(u^{n_{k}})^2+2(u^{n_{k}}_{x})^2\}$ and $(u^{n_{k}}_{x})^2$ are uniformly bounded in $L^{2}(\mathbb{R})$.
Thus, there exists subsequences, denoted again  $\{6(u^{n_{k}})^2+2(u^{n_{k}}_{x})^2\}$ and $(u^{n_{k}}_{x})^2$,
which converge weakly in $L^{2}(\mathbb{R}).$
Thank to (\ref{l2}) and (\ref{l02}), we deduce for a.e. $t\in (0,T)$
\begin{align}
\{6[u^{n_{k}}]^2+2[u^{n_{k}}_{x}]^2\} &\rightharpoonup 6u^2+2u_{x}^2~~~~ \text{weakly~ ~in}~~ L^{2}(\mathbb{R})~ ~\text{for} ~~ n_{k}\rightarrow\infty,
\\
[u^{n_{k}}_{x}]^2 &\rightharpoonup u_{x}^2~~~~ \text{weakly~ ~in}~~ L^{2}(\mathbb{R})~ ~\text{for} ~~ n_{k}\rightarrow\infty.
\end{align}

In view of $G,~G_{x}\in L^{2}(\mathbb{R})$, we obtain
 \begin{align}
G_{x}\ast\{6[u^{n_{k}}]^2+2[u^{n_{k}}_{x}]^2\} & \rightarrow G_{x}\ast(6u^2+2u_{x}^2)~~~~ a.e.~\text{on} ~(0,T)\times\mathbb{R}~ ~\text{as} ~~ n_{k}\rightarrow\infty,
\\
G\ast[u^{n_{k}}_{x}]^2 & \rightarrow G\ast[u_{x}^2]~~~~a.e.~\text{on}~ (0,T)\times\mathbb{R}~ ~\text{as} ~~ n_{k}\rightarrow\infty,
\end{align}

which along with (\ref{l2})-(\ref{l02}) and (\ref{l3}), implies that $u$ satisfies Eq. (\ref{E1}) in $\mathcal{D}'((0,T)\times\mathbb{R})$.

 {\bf Step 3:}
We prove that $u\in C_{w}(\mathbb{R}_{+}; H^{2}(\mathbb{R}))$.
Due to (\ref{029}), we have an uniform bound on $\|u^{n_{k}}(t,\cdot)\|_{H^2(\mathbb{R})}$ for all $t\in \mathbb{R}_{+}$ and all $n_{k}$, which together with (\ref{E2}) leads to
\begin{align}
\nonumber\frac{d}{dt}&\int_{\mathbb{R}}u^{n_{k}}(t,x)\varphi(x)dx\\\nonumber
=&\int_{\mathbb{R}}u_{t}^{n_{k}}(t,x)\varphi(x)dx
\\\nonumber
=&\int_{\mathbb{R}}\bigg(4u^{n_{k}}u^{n_{k}}_x+G\ast\{\partial_x[2(u^{n_{k}}_x)^2+6 (u^{n_{k}})^{2}]+\partial^2_x(u^{n_{k}}_{x})^2\}\bigg)\varphi dx
\\\nonumber
=&\int_{\mathbb{R}}2\partial_{x}(u^{n_{k}})^2\varphi dx
+\int_{\mathbb{R}}G\ast\{\partial_x[2(u^{n_{k}}_x)^2+6 (u^{n_{k}})^{2}]\}\varphi dx+
\int_{\mathbb{R}}G\ast[\partial^2_x(u^{n_{k}}_{x})^2]\varphi dx
\\\nonumber
=&-\int_{\mathbb{R}}2(u^{n_{k}})^2\varphi_{x} dx
-\int_{\mathbb{R}}G\ast[2(u^{n_{k}}_x)^2+6 (u^{n_{k}})^{2}]\varphi_{x} dx+
\int_{\mathbb{R}}G\ast[(u^{n_{k}}_{x})^2]\varphi_{xx} dx
\\\nonumber
\leq&
2\|(u^{n_{k}})^2\|_{H^1}\|\varphi_{x}\|_{H^{-1}}+\|G\star[2(u^{n_{k}}_x)^2+6 (u^{n_{k}})^{2}]\|_{H^1}\|\varphi_{x}\|_{H^{-1}}
+\|G\ast[(u^{n_{k}}_{x})^2]\|_{L^2}\|\varphi_{xx}\|_{L^2}
\\\nonumber\leq&
4\|u^{n_{k}}\|_{H^1}\|u^{n_{k}}\|_{L^\infty}\|\varphi_{x}\|_{H^{-1}}
+\|G\ast[2(u^{n_{k}}_x)^2+6 (u^{n_{k}})^{2}]\|_{L^{2}}\|\varphi_{x}\|_{H^{-1}}
\\\nonumber&+\|G_{x}\ast[2(u^{n_{k}}_x)^2+6 (u^{n_{k}})^{2}]\|_{L^{2}}\|\varphi_{x}\|_{H^{-1}}
+\|G\ast[(u^{n_{k}}_{x})^2]\|_{L^2}\|\varphi_{xx}\|_{L^2}
\\\nonumber\leq&
4\|u^{n_{k}}\|^2_{H^1}\|\varphi_{x}\|_{H^{-1}}
+\|G\|_{L^{2}}\|2(u^{n_{k}}_x)^2+6 (u^{n_{k}})^{2}\|_{L^{1}}\|\varphi_{x}\|_{H^{-1}}
\\\nonumber&+\|G_{x}\|_{L^{2}}\|2(u^{n_{k}}_x)^2+6 (u^{n_{k}})^{2}\|_{L^{1}}\|\varphi_{x}\|_{H^{-1}}
+\|G\star[(u^{n_{k}}_{x})^2]\|_{L^2}\|\varphi_{xx}\|_{L^2}
\\\nonumber\leq&(12\|\varphi_{x}\|_{H^{-1}}
+\|\varphi_{xx}\|_{L^2})\|u_{0}\|^2_{H^1(\mathbb{R})}
\\\nonumber\leq& C\|u_{0}\|^2_{H^1(\mathbb{R})},
\end{align}
where $\varphi(x)\in H^{-2}$.
Hence the family $t\mapsto u^{n_{k}}(t, \cdot)\in H^{2}(\mathbb{R})$ is weakly equicontinuous on $[0,T]$ for any $T>0$. Making use of the Arzela-Ascoli theorem, implies that $\{u^{n_{k}}\}$ contains a subsequence, denoted again by $\{u^{n_{k}}\}$, converging weakly in $H^2(\mathbb{R})$, uniformly in $t$. The previous discussion concludes the fact that the limit function $u$ is weakly continuous from $\mathbb{R}_{+}$ into $H^2(\mathbb{R}).$

Because of the weakly convergence in $H^2(\mathbb{R})$ for a.e. $t\in \mathbb{R}_{+}$, we have
\begin{align}
\nonumber\|u(t,\cdot)\|_{H^2(\mathbb{R})}&\leq \liminf_{n_{k}\rightarrow\infty}\|u^{n_{k}}(t,\cdot)\|_{H^2(\mathbb{R})} =\liminf_{n_{k}\rightarrow\infty}[3\|u_{0}^{n}\|_{H^1(\mathbb{R})}+2\|(2-\partial_{x})m_{0}^{n}\|_{H^1(\mathbb{R})}]
\\\nonumber&\leq 3\|u_{0}\|_{H^2(\mathbb{R})}+2\|2m_{0}-m_{0}'\|_{\mathcal{M}},
\end{align}

for a.e. $t\in \mathbb{R}_{+}$. Therefore $u\in L^{\infty}(\mathbb{R}_{+}\times\mathbb{R})$, which along with Theorem \ref{thm}, we get
\begin{align}
\nonumber\|u^{n}_{x}(t,\cdot)\|_{L^\infty(\mathbb{R})}&\leq \|u^{n}(t,\cdot)\|_{H^2(\mathbb{R})}
\\\nonumber&\leq3\|u_{0}\|_{H^2(\mathbb{R})}+2\|2m_{0}-m_{0}'\|_{\mathcal{M}},
\\\nonumber
\|u^{n}_{xx}(t,\cdot)\|_{L^\infty(\mathbb{R})}&\leq3\|u^{n}_{x}(t,\cdot)\|_{L^\infty(\mathbb{R})}
+2\|u^{n}(t,\cdot)\|_{L^\infty(\mathbb{R})}
\\\nonumber&\leq3\|u^{n}_{x}(t,\cdot)\|_{L^\infty(\mathbb{R})}+2\|u^{n}\|_{H^1(\mathbb{R})},
\\\nonumber&\leq6\|2m_{0}-m_{0}'\|_{\mathcal{M}}+11\|u_{0}\|_{H^1(\mathbb{R})},~~n\geq1,
\end{align}
which along with (\ref{l3}), leads to  $u_{x}\in L^{\infty}(\mathbb{R}_{+}\times\mathbb{R})$ and $u_{xx}\in L^{\infty}(\mathbb{R}_{+}\times\mathbb{R}).$

 {\bf Step 4:}
We conclude that $u\in C(\mathbb{R}_{+}; H^{1}(\mathbb{R}))$.
In order to prove $u\in C(\mathbb{R}_{+}; H^{1}(\mathbb{R}))$, we need to prove that the functional $E(u(t))= \|u(t,\cdot)\|^2_{H^1(\mathbb{R})}$ is conserved in time.

Indeed, if this holds, then
\begin{align}
\|u(t)-u(s)\|^{2}_{H^{1}(\mathbb{R})}&=\|u(t)\|^2_{H^{1}(\mathbb{R})}
+\|u(s)\|^2_{H^{1}(\mathbb{R})}-2\bigg(u(t),u(s)\bigg)_{H^{1}}
\nonumber\\\nonumber&=2\|u_{0}\|^2_{H^{1}(\mathbb{R})}-2\bigg(u(t),u(s)\bigg)_{H^{1}}, ~~t,s ~\in \mathbb{R}_{+}.
\end{align}

The scalar product $(u(t),u(s))_{H^{1}}$ converges to $\|u(t)\|^2_{H^{1}(\mathbb{R})}=\|u_{0}\|^2_{H^{1}(\mathbb{R})},$ as $s\rightarrow t$.

Thus, we get $\|u(t)-u(s)\|_{H^{1}(\mathbb{R})}\rightarrow 0$,  as $s\rightarrow t$. That is $u\in C(\mathbb{R}_{+}; H^{1}(\mathbb{R}))$.

We prove the conservation of $E(u)$ in time by a regularization technique.

Since $u$ satisfies (\ref{E2}) in distributional sense, it follows that for a.e. $t \in \mathbb{R}_{+}$
\begin{align}\label{l5}
\rho_{n}\ast u_t-4\rho_{n}\ast(uu_x)-\rho_{n}\ast [G_{x}\ast(2u_x^2+6 u^{2})+\partial_{x}^2G\ast(u^2_{x})]=0, ~~ n>1.
\end{align}

Taking an $L^2$ energy estimate with $\rho_{n}\ast u$ yields for a.e. $t \in \mathbb{R}_{+}$
\begin{align}\label{l6}
\frac{1}{2}\frac{d}{dt}\int_{\mathbb{R}}(\rho_{n}\ast u)^2dx&-4\int_{\mathbb{R}}(\rho_{n}\ast u) [\rho_{n}\ast(uu_x)]dx
\nonumber\\ &-\int_{\mathbb{R}}(\rho_{n}\ast u)\bigg\{\rho_{n}\ast [G_{x}\ast(2u_x^2+6 u^{2})
+\partial^2_{x}G\ast(u^2_{x})]\bigg\}dx=0.
\end{align}

Differentiating (\ref{l5}) with respect to $x$, then taking $L^2$ inner production with $\rho_{n}\ast u_{x}$, we obtain
\begin{align}
\frac{1}{2}\frac{d}{dt}\int_{\mathbb{R}}(\rho_{n}\ast u_{x})^2dx&-4\int_{\mathbb{R}}(\rho_{n}\ast u_{x}) [\rho_{n,x}\ast(uu_x)]dx
\nonumber\\ \nonumber&-\int_{\mathbb{R}}(\rho_{n}\ast u_{x})[\rho_{n}\ast\partial^2_{x}G\ast(2u_x^2+6 u^{2})+\rho_{n,x}\ast\partial^2_{x}G\ast(u^2_{x})]dx=0.
\end{align}

Since $\partial_{x}^2(G\ast f)=G\ast f-f$ for $f\in L^2(\mathbb{R})$, it follows from $2u_x^2+6 u^{2}\in L^2(\mathbb{R})~and~u^2_{x}\in L^2(\mathbb{R})$ that
\begin{align}\label{l7}
\frac{1}{2}\frac{d}{dt}\int_{\mathbb{R}}(\rho_{n}\ast u_{x})^2dx&-4\int_{\mathbb{R}}(\rho_{n}\ast u_{x}) [\rho_{n,x}\ast(uu_x)]dx
\nonumber\\ \nonumber&-\int_{\mathbb{R}}(\rho_{n}\ast u_{x})[\rho_{n}\ast G\ast(2u_x^2+6 u^{2})]dx-\int_{\mathbb{R}}(\rho_{n}\ast u_{x})[\rho_{n}\ast G_x\ast(u^2_{x})]dx
\\&+\int_{\mathbb{R}}(\rho_{n}\ast u_{x})[\rho_{n}\ast(2u_x^2+6 u^{2})+\rho_{n,x}\ast(u_x^2)]dx=0.
\end{align}

Combining (\ref{l6}) and (\ref{l7}) yields
\begin{align}
\nonumber\frac{1}{2}\frac{d}{dt}\int_{\mathbb{R}}[(\rho_{n}\ast u)^2+(\rho_{n}\ast u_{x})^2]dx
=&-2\int_{\mathbb{R}}(\rho_{n}\ast u_{x})[\rho_{n}\ast(u_x^2)]dx-8\int_{\mathbb{R}}(\rho_{n}\ast u_{x})[\rho_{n}\ast( u^{2})]dx\\&-\int_{\mathbb{R}}(\rho_{n}\ast u_{x})[\rho_{n,x}\ast(u_x^2)]dx+4\int_{\mathbb{R}}(\rho_{n}\ast u_{x})[\rho_{n,x}\ast(uu_x)]dx.
\end{align}

Note that
\begin{align}
\lim_{n\rightarrow{\infty}}\|\rho_{n}\ast u_{x}-u_{x}\|_{L^2(\mathbb{R})}&=\lim_{n\rightarrow{\infty}}\|\rho_{n}\ast u^2_{x}-u^2_{x}\|_{L^2(\mathbb{R})}\nonumber=\lim_{n\rightarrow{\infty}}\|\rho_{n}\ast u^2-u^2\|_{L^2(\mathbb{R})}
\\ \nonumber
&=\lim_{n\rightarrow{\infty}}\|\rho_{n}\ast u_{xx}-u_{xx}\|_{L^2(\mathbb{R})}
=\lim_{n\rightarrow{\infty}}\|\rho_{n}\ast u-u\|_{L^2(\mathbb{R})}
=0.
\end{align}

Therefore
\begin{align}
&\int_{\mathbb{R}}(\rho_{n}\ast u_{x})[\rho_{n}\ast( u^{2})]dx\rightarrow 0~~~\text{as} ~n\rightarrow\infty,
\\
&\int_{\mathbb{R}}(\rho_{n}\ast u_{x})[\rho_{n}\ast(u_x^2)]dx\rightarrow \int_{\mathbb{R}}u_x^3dx ~~~\text{as} ~n\rightarrow\infty,
\\&
\int_{\mathbb{R}}[\rho_{n,x}\ast (u^2_{x})](\rho_{n}\ast u_{x})dx=-\int_{\mathbb{R}}(\rho_{n,xx}\ast u)
(\rho_{n}\ast u^2_{x})dx\rightarrow \int_{\mathbb{R}}u_x^2u_{xx}dx=0 ~~~\text{as} ~n\rightarrow\infty,
\end{align}
for $u(t,\cdot)\in H^2(\mathbb{R})$ for a.e. $t\in\mathbb{R}_{+}$.

Observe that
\begin{align}
\nonumber\int_{\mathbb{R}}[\rho_{n,x}\ast (u u_{x})]&(\rho_{n}\ast u_{x})dx=-\int_{\mathbb{R}}(\rho_{n,xx}\ast u )(\rho_{n}\ast u_{x})(\rho_{n}\ast u)dx\\&+\int_{\mathbb{R}}[\rho_{n,xx}\ast u][(\rho_{n}\ast u)(\rho_{n}\ast u_{x})
-\rho_{n}\ast(uu_{x})]dx.
\end{align}

In view of Lemma \ref{lem9} with $u\in W^{1,\infty}(\mathbb{R})$ and
$u(t,)\in H^1({\mathbb{R}})\subset L^{2}(\mathbb{R})$, we deduce
\begin{align}
\|[\rho_{n}&\ast(u u_{x})-(\rho_{n}\ast u_{x})(\rho_{n}\ast u)]\|_{L^2(\mathbb{R})}
\\\nonumber \leq&\|\rho_{n}\ast(u u_{x})-u(\rho_{n}\ast u_{x})\|_{L^2(\mathbb{R})}
\\\nonumber&+\|u\rho_{n}\ast(u_{x})-(\rho_{n}\ast u_{x})(\rho_{n}\ast u)\|_{L^2(\mathbb{R})}
\\\nonumber \leq&\|\rho_{n}\ast(u u_{x})-u(\rho_{n}\ast u_{x})\|_{L^2(\mathbb{R})}
\\\nonumber&+\|\rho_{n}\ast u_{x}\|_{L^\infty(\mathbb{R})}\|u-\rho_{n}\ast u\|_{L^2(\mathbb{R})}
\rightarrow 0,~~ ~as~ ~n \rightarrow \infty,
\end{align}

which along with the fact that
\begin{align}
\nonumber\|\rho_{n,xx}\ast u\|_{L^2(\mathbb{R})}\leq& \|u_{xx}\|_{L^2(\mathbb{R})}
 \leq \|u\|_{H^2(\mathbb{R})}
\\ \nonumber  \leq&\|u_{0}\|_{H^1(\mathbb{R})}+2\|2m_{0}-m_{0}'\|_{\mathcal{M}},
\end{align}
implies
\begin{align}
\int_{\mathbb{R}}[\rho_{n,xx}\ast u][(\rho_{n}\ast u)(\rho_{n}\ast u_{x})
-\rho_{n}\ast(uu_{x})]dx \rightarrow 0 ~~~ ~as~ ~n \rightarrow \infty.
\end{align}

On the other hand, an integration by parts leads to
$$
\int_{\mathbb{R}}[\rho_{n,xx}\ast u](\rho_{n}\ast u_{x})(\rho_{n}\ast u)dx=-\frac{1}{2}\int_{\mathbb{R}}(\rho_{n}\ast u_{x})^3 dx \rightarrow -\frac{1}{2}\int_{\mathbb{R}}( u_{x})^3 dx  ~~~~as~ ~n \rightarrow \infty.
$$

Now we define
\begin{align}
&\nonumber E_{n}(t):=\int_{\mathbb{R}}[(\rho_{n}\ast u)^2+(\rho_{n}\ast u_{x})^2]dx,~~~t\in \mathbb{R}_{+},~~n\leq~1.
\\\nonumber&
G_{n}=-2\int_{\mathbb{R}}(\rho_{n}\ast u_{x})[\rho_{n}\ast(u_x^2)]dx-8\int_{\mathbb{R}}(\rho_{n}\ast u_{x})[\rho_{n}\ast( u^{2})]dx\\\nonumber&-\int_{\mathbb{R}}(\rho_{n}\ast u_{x})[\rho_{n,x}\ast(u_x^2)]dx+4\int_{\mathbb{R}}(\rho_{n}\ast u_{x})[\rho_{n,x}\ast(uu_x)]dx, ~~t\in \mathbb{R}_{+},~~n\geq1.
\end{align}

The previous discussion enables us to get
\begin{align}\label{l8}
\frac{d}{dt}E_{n}(t)=G_{n}(t),~~~~~~~n\geq1,
\end{align}
and
\begin{align}
G_{n}(t)~\rightarrow 0,~ ~\text{uniformly~in}~~t,~~~\text{as}~ ~n \rightarrow \infty,
\end{align}
for a.e. $t\in\mathbb{R}_{+}$.

Taking advantage of Lemma \ref{lem1} and (\ref{l8}), we get
\begin{align}
E_{n}(t)-E_{n}(0)=\int_{0}^{t}G_{n}(s)ds,~~t\in\mathbb{R}_{+},~~n\geq1.
\end{align}

Making full use of H\"{o}lder's inequality and Young's inequality, we obtain the existence of a constant $K >0$ which only depends on $\|u_{0}\|_{H^1(\mathbb{R})}$ and $\|2m_{0}-m_{0}'\|_{\mathcal{M}}$
such that
\begin{align}\label{l9}
|G_{n}(t)|\leq K; ~~~~t \in \mathbb{R}_{+};~~~ n \leq 1,
\end{align}

for $u,u_{x}\in L^{\infty}(\mathbb{R}_{+}\times\mathbb{R}).$

Combining with (\ref{l8})-(\ref{l9}), by Lebesgue's dominated convergence theorem, we term to
$$\lim_{n\rightarrow 0}[E_{n}(t)-E_{n}(0)]=0,~~~t \in \mathbb{R}_{+}.$$

Therefore, by (\ref{l10}), we have
$$E(u(t))=\lim_{n\rightarrow 0}E_{n}(t)=E(u_{0}),$$
for fix $t \in \mathbb{R}_{+}$,

This implies that $ u \in C(\mathbb{R}_{+};H^1(\mathbb{R}))$ and $E$ is conserved along our solution.

In view of (\ref{E1}) and H\"{o}lder's inequality and Young's inequality, we have $ u \in C^1(\mathbb{R}_{+};L^2(\mathbb{R}))$.

 {\bf Step 5:}
Finally, we prove the uniqueness of weak solutions to (\ref{E2}).
Let $u~ and~ v $ be two weak solutions of (\ref{E2}) within the class $\{f~\in ~ C_{w}(\mathbb{R}_{+};H^2(\mathbb{R}))\cap C(\mathbb{R}_{+};H^1(\mathbb{R}))\cap C^1(\mathbb{R}_{+};L^2(\mathbb{R}))$ with the total variation of $f-f_{xx}\in\mathcal{M}(\mathbb{R})$ uniformly bounded on
$\mathbb{R}_{+}$\}.

Setting
$M:=\sup_{t\geq0} \{\|u(t,\cdot)-u_{xx}(t,\cdot)\|_{\mathcal{M}}+\|v(t,\cdot)-v_{xx}(t,\cdot)\|_{\mathcal{M}}\}$, we have
\begin{align}\label{l13}
\nonumber|u(t,x)|&=|G\ast [u-u_{xx}](x)|
\\ &\leq \|G\|_{L^{\infty}} \|[u-u_{xx}](t,\cdot)\|_{\mathcal{M}}\leq\frac{1}{2}M,
\\|u_{x}(t,x)|&=|G_{x}\ast [u-u_{xx}](x)|\leq\frac{1}{2}M,
\\|u_{xx}(t,x)|&=|(2u-u_{x})(t,x)|+2|u_{x}(t,x)|\leq\frac{5}{2}M,
\end{align}
for all $(t,x)\in \mathbb{R}_{+}\times  \mathbb{R}$.
Similarly,
\begin{align}
|v(t,x)|\leq\frac{1}{2}M,~~|v_{x}(t,x)|\leq\frac{1}{2}M,~~~(t,x)\in \mathbb{R}_{+}\times  \mathbb{R}.
\end{align}
Applying the fact $u-u_{xx},~v-v_{xx}\in\mathcal{M}_{+}(\mathbb{R})$ and
\begin{align}
\nonumber\|G\ast (u-u_{xx})\|_{L^1}=&\sup_{\varphi \in L^{\infty}(\mathbb{R}),\|\varphi\|_{L^{\infty}}\leq1}
\int_{\mathbb{R}}\varphi(y)[G\ast (u-u_{xx})](y) dy
\\\nonumber&=\sup_{\varphi\in L^{\infty}(\mathbb{R}),\|\varphi\|_{L^{\infty}}\leq1}
\int_{\mathbb{R}}\varphi(y)\int_{\mathbb{R}}G(y-x) d(u-u_{xx})(x) dy
\\\nonumber&=\sup_{\varphi\in L^{\infty}(\mathbb{R}),\|\varphi\|_{L^{\infty}}\leq1}
\int_{\mathbb{R}}(G\ast \varphi) d(u-u_{xx})(x)
\\\nonumber&=\sup_{\varphi\in L^{\infty}(\mathbb{R}),\|\varphi\|_{L^{\infty}}\leq1}
\|G\ast \varphi\|_{L^{\infty}(\mathbb{R})} \|u-u_{xx}\|_{\mathcal{M}}
\\\nonumber&\leq\sup_{\varphi \in L^{\infty}(\mathbb{R}),\|\varphi\|_{L^{\infty}}\leq1}
\|G\|_{L^1(\mathbb{R})}\|\varphi\|_{L^{\infty}(\mathbb{R})} \|u-u_{xx}\|_{\mathcal{M}}
\\\nonumber&=\|u-u_{xx}\|_{\mathcal{M}}.
\end{align}

Thus, we obtain for $t\geq1$
\begin{align}
\nonumber&\|u(t,\cdot)\|_{L^1}=\|G\ast [u(t,\cdot)-u_{xx}(t,\cdot)]\|_{L^1}\leq M.
\end{align}
Similarly, we infer for $t\geq1$
\begin{align}
&\|u_{x}(t,\cdot)\|_{L^1}=\|G_{x}\ast [u(t,\cdot)-u_{xx}(t,\cdot)]\|_{L^1}\leq M,
\\&\|v(t,\cdot)\|_{L^1}=\|G\ast [v(t,\cdot)-v_{xx}(t,\cdot)]\|_{L^1}\leq M,
\end{align}
\begin{align}\label{l14}
\|v_{x}(t,\cdot)\|_{L^1}=\|G_{x}\ast [v(t,\cdot)-v_{xx}(t,\cdot)]\|_{L^1}\leq M.
\end{align}
Let us define
$$w(t,x):=u(t,x)-v(t,x),~~~(t,x)\in \mathbb{R}_{+}\times \mathbb{R}.$$
Thus we obtain
\begin{align}\label{E4}
\left\{
\begin{array}{ll}
w_t(t,x)=&4vw_x+4u_{x}w-(u_{x}+v_{x})w_{x}+2G_{x}\ast[(u_{x}+v_{x})w_{x}]
\\&+6G_{x}\ast[(u+v)w]+G\ast[(u_{x}+v_{x})w_{x}], ~~ t>0,\\[1ex]
w(0,x)=&u_{0}(x)-v_{0}(x)=w_{0}(x).
\end{array}
\right.
\end{align}

We claim that  $\eta :\mathbb{R}_{+}\rightarrow \mathbb{R}$ is a decreasing
$C^2$ function such that $\eta(s) = 1$ for $s\in[0,\frac{1}{2}]$;  $\eta(s) = e^{-s}$ for $s\in[1,\infty]$
 and $\eta$ is a polynomial on $[\frac{1}{2},1].$

For $R > 0$, we put
$$\eta_{R}(x) = \eta(\frac{|x|}{R}),~~x\in\mathbb{R}.$$

Thanks to Lemma \ref{lem1}, we have for all $t\geq0$ and $n\geq1,$
\begin{align}
\int_{\mathbb{R}}|\rho_{n}\ast w_{x}|(t,x)\eta_{R}(x)dx&-\int_{\mathbb{R}}|\rho_{n}\ast w_{x}|(0,x)\eta_{R}(x)dx\nonumber\\&=\int_{0}^{t}\int_{\mathbb{R}}\frac{\partial|\rho_{n}\ast w_{x}|(s,x)}{\partial_{s}}\eta_{R}(x)dxds.
\end{align}

Notice that
$$ \partial_{t}|\rho_{n}\ast w_{x}|=(\rho_{n}\ast w_{tx})sgn(\rho_{n}\ast w_{x}).$$
Then $t\mapsto \rho_{n}\ast w_{tx}=\rho_{n,x}\ast w$ is uniformly bounded in $L^1(\mathbb{R})$, since
\begin{align}
\nonumber\|\rho_{n}&\ast w_{tx}\|_{L^1(\mathbb{R})}
\\ \nonumber\leq&\|\rho_{n}\ast \partial_{x} \{4vw_x+4u_{x}w-(u_{x}+v_{x})w_{x}
\\\nonumber&+G_{x}\ast[2(u_{x}+v_{x})w_{x}+6(u+v)w]
+G\ast[(u_{x}+v_{x})w_{x}]\}\|_{L^1(\mathbb{R})}
\\\nonumber\leq&
\|\rho_{n}\ast[(2v_{x}-4v)(w-w_{xx})]\|_{L^1(\mathbb{R})}
+\|[-2u_{xx}+2(u_{x}+v_{x})]w_{x}\|_{L^1(\mathbb{R})}
\\\nonumber&+\|[4u_{xx}-2v_{x}-2v-6u]w\|_{L^1(\mathbb{R})}
+\|2G\ast[(u_{x}+v_{x})w_{x}]\|_{L^1(\mathbb{R})}
\\\nonumber&+\|6G\ast[(u+v)w]\|_{L^1(\mathbb{R})}+
\|G_{x}\ast[(u_{x}+v_{x})w_{x}]\|_{L^1(\mathbb{R})}
\\ \nonumber
\leq&
\|2v_{x}-4v\|_{L^\infty(\mathbb{R})}\|w-w_{xx}\|_{\mathcal{M}(\mathbb{R})}
+\|2(u_{x}+v_{x})\|_{L^\infty(\mathbb{R})}\|w_{x}\|_{L^1(\mathbb{R})}
\\\nonumber&+\|[-2u_{xx}+2(u_{x}+v_{x})]\|_{L^\infty(\mathbb{R})}\|w_{x}\|_{L^1(\mathbb{R})}
+\|w\|_{L^1(\mathbb{R})}\|[4u_{xx}-2v_{x}-2v-6u]\|_{L^\infty(\mathbb{R})}
\\\nonumber&+\|6(u+v)\|_{L^\infty(\mathbb{R})}\|w\|_{L^1(\mathbb{R})}+
\|u_{x}+v_{x}\|_{L^\infty(\mathbb{R})}\|w_{x}\|_{L^1(\mathbb{R})}
\\ \leq&34M^2.
\end{align}

Setting $R\rightarrow\infty,$ applying Lebesgue's dominated convergence theorem, yields that for
all $t\in\mathbb{R}_{+},$

$$\int_{\mathbb{R}}|\rho_{n}\ast w_{x}|(t,x)dx-\int_{\mathbb{R}}|\rho_{n}\ast w_{x}|(0,x)dx=\int_{0}^{t}\int_{\mathbb{R}}\partial_{s}|\rho_{n}\ast w_{x}|(s,x)dxds.
$$
Differentiating the above relation with respect to time, we propose that
\begin{align}\label{l32}
\frac{d}{dt}\int_{\mathbb{R}}|\rho_{n}\ast w_{x}|dx
=\int_{\mathbb{R}}(\rho_{n}\ast w_{tx})sgn(\rho_{n}\ast w_{x})dx.
\end{align}

Similarly, we infer that
\begin{align}\label{l31}
\frac{d}{dt}\int_{\mathbb{R}}|\rho_{n}\ast w|dx
=\int_{\mathbb{R}}(\rho_{n}\ast w_{t})sgn(\rho_{n}\ast w)dx,
\end{align}

Therefore, convoluting $\rho_{n} $ with $w$, by (\ref{E4}) and (\ref{l31}), we get
\begin{align}\label{l12}
\frac{d}{dt}&\int_{\mathbb{R}}|\rho_{n}\ast w|dx
\nonumber \\ \nonumber=&\int_{\mathbb{R}}(\rho_{n}\ast w_{t})sgn(\rho_{n}\ast w)dx
\\ \nonumber=&\int_{\mathbb{R}}[\rho_{n}\ast (4vw_x)]sgn(\rho_{n}\ast w)dx+\int_{\mathbb{R}}[\rho_{n}\ast (4u_{x}w)]sgn(\rho_{n}\ast w)dx
\\\nonumber&-\int_{\mathbb{R}}\bigg\{\rho_{n}\ast [(u_{x}+v_{x})w_{x}]\bigg\}sgn(\rho_{n}\ast w)dx+\int_{\mathbb{R}}\bigg\{\rho_{n}\ast G\ast[(u_{x}+v_{x})w_{x}]\bigg\}sgn(\rho_{n}\ast w)dx
\\\nonumber&+\int_{\mathbb{R}}\bigg\{\rho_{n}\ast G_{x}\ast[2(u_{x}+v_{x})w_{x}+6(u+v)w]\bigg\}sgn(\rho_{n}\ast w)dx
\\ \nonumber\leq&\int_{\mathbb{R}}|\rho_{n}\ast (4vw_x)|dx+\int_{\mathbb{R}}|\rho_{n}\ast (4u_{x}w)|dx
\\\nonumber&+\int_{\mathbb{R}}|\rho_{n}\ast [(u_{x}+v_{x})w_{x}]|dx
+\int_{\mathbb{R}}|\rho_{n}\ast G\ast[(u_{x}+v_{x})w_{x}]|dx
\\&+\int_{\mathbb{R}}|\rho_{n}\ast G_{x}\ast[2(u_{x}+v_{x})w_{x}+6(u+v)w]|dx.
\end{align}

We estimate the terms on the right-hand side of (\ref{l12}), respectively.
Due to (\ref{l13})-(\ref{l14}), it follows that
\begin{align}\label{l18}
\int_{\mathbb{R}}|\rho_{n}\ast (4vw_x)|dx
\nonumber\leq&4\int_{\mathbb{R}}(\rho_{n}\ast |vw_x|)dx
\\\nonumber\leq&4\|v\|_{L^{\infty}}\int_{\mathbb{R}}(\rho_{n}\ast |w_x|)dx
\\\leq&2M\int_{\mathbb{R}}|\rho_{n}\ast w_x|dx
+2M\int_{\mathbb{R}}\bigg(\rho_{n}\ast |w_{x}|-|\rho_{n}\ast w_x|\bigg)dx,
\\
\nonumber
\int_{\mathbb{R}}|\rho_{n}\ast (4u_{x}w)|dx
\leq&
\int_{\mathbb{R}}\rho_{n}\ast |4u_{x}w|dx
\\ \nonumber
\leq&4\|u_{x}\|_{L^{\infty}}\int_{\mathbb{R}}\rho_{n}\ast |w|dx
\\ \leq&2M\int_{\mathbb{R}}|\rho_{n}\ast w|dx
+2M\int_{\mathbb{R}}\bigg(\rho_{n}\ast |w|-|\rho_{n}\ast w|\bigg)dx,
\\\nonumber
\int_{\mathbb{R}}|\rho_{n}\ast [(u_{x}+v_{x})w_{x}]|dx
\leq&\int_{\mathbb{R}}\rho_{n}\ast |(u_{x}+v_{x})w_{x}|dx
\\\nonumber\leq&\|u_{x}+v_{x}\|_{L^{\infty}}\int_{\mathbb{R}}\rho_{n}\ast |w_{x}|dx
\\ \leq &M\int_{\mathbb{R}}|\rho_{n}\ast w_{x}|dx
+M\int_{\mathbb{R}}\bigg(\rho_{n}\ast |w_{x}|-|\rho_{n}\ast w_x|\bigg)dx,
\\\nonumber
\int_{\mathbb{R}}|\rho_{n}\ast G\ast[(u_{x}+v_{x})w_{x}]|dx
\leq&\int_{\mathbb{R}}\bigg|\int_{\mathbb{R}}\int_{\mathbb{R}}
\rho_{n}(y)G(x-z-y)[(u_{x}+v_{x})w_{x}](z)|dydz\bigg|dx.
\end{align}
Setting $w=z+y$, we get
\begin{align}\label{l11}
\nonumber\int_{\mathbb{R}}\bigg|\int_{\mathbb{R}}&
\int_{\mathbb{R}}\rho_{n}(y)G(x-z-y)[(u_{x}+v_{x})w_{x}](z)dydz\bigg|dx
\\\nonumber=&\int_{\mathbb{R}}\bigg|\int_{\mathbb{R}}\int_{\mathbb{R}}
\rho_{n}(w-z)G(x-w)[(u_{x}+v_{x})w_{x}](z)
dwdz\bigg|dx
\\\nonumber=&\int_{\mathbb{R}}\bigg|\int_{\mathbb{R}}\rho_{n}\ast [(u_{x}+v_{x})w_{x}](w)G(x-w)dy\bigg|dx
\\\nonumber=&\int_{\mathbb{R}}|G\ast\{\rho_{n}\ast [(u_{x}+v_{x})w_{x}]\}|dx
\\\nonumber\leq&\|G\|_{L^1}\int_{\mathbb{R}}|\rho_{n}\ast [(u_{x}+v_{x})w_{x}]|dx
\\\nonumber\leq&\|u_{x}+v_{x}\|_{L^\infty}\int_{\mathbb{R}}\rho_{n}\ast |w_{x}|dx
\\\leq&M\int_{\mathbb{R}}|\rho_{n}\ast w_{x}|dx
+M\int_{\mathbb{R}}\bigg(\rho_{n}\ast |w_{x}|-|\rho_{n}\ast w_{x}|\bigg)dx.
\end{align}

By the same token with the last inequality, we obtain
\begin{align}\label{l19}
\nonumber\int_{\mathbb{R}}&|\rho_{n}\ast G_{x}\ast[2(u_{x}+v_{x})w_{x}+6(u+v)w]|dx
\\\nonumber\leq&\|G_{x}\|_{L^1}\int_{\mathbb{R}}|\rho_{n}\ast [2(u_{x}+v_{x})w_{x}+6(u+v)w]|dx
\\\nonumber\leq&2\|u_{x}+v_{x}\|_{L^\infty}\int_{\mathbb{R}}\rho_{n}\ast |w_{x}|dx
+6\|u+v\|_{L^\infty}\int_{\mathbb{R}}\rho_{n}\ast |w|dx
\\\nonumber\leq&2M\int_{\mathbb{R}}|\rho_{n}\ast w_{x}|dx+6M\int_{\mathbb{R}}|\rho_{n}\ast w|dx
\\&+2M\int_{\mathbb{R}}\bigg(\rho_{n}\ast |w_{x}|-|\rho_{n}\ast w_{x}|\bigg)dx
+6M\int_{\mathbb{R}}\bigg(\rho_{n}\ast |w|-|\rho_{n}\ast w|\bigg)dx.
\end{align}

Thus, combining (\ref{l18})-(\ref{l19}) with (\ref{l12}), we have
\begin{align}\label{l23}
\frac{d}{dt}&\int_{\mathbb{R}}|\rho_{n}\ast w|dx\leq 6M\int_{\mathbb{R}}|\rho_{n}\ast w_x|dx+8M\int_{\mathbb{R}}|\rho_{n}\ast w|dx+Q_{n}(t).
\end{align}

Here $Q_{n}(t)$ satisfies the follow two conditions:
\begin{align}\label{l21}
\left\{
\begin{array}{ll}
&(1)~~~~Q_{n}(t)\rightarrow0~~\text{uniformly~in}~t,~~\text{as}~~ t\rightarrow\infty,\\[1ex]
&(2)~~~~|Q_{n}(t)|\leq H,~~ n\geq1,~~t\in\mathbb{R}_{+},
\end{array}
\right.
\end{align}
where $H$ only depends on $M$.

Then, convoluting $\rho_{n,x} $ with $w$, by (\ref{l32}) and differentiating (\ref{E4}) with respect to $x$, we obtain

\begin{align}\label{l17}
\frac{d}{dt}\int_{\mathbb{R}}|\rho_{n}\ast w_{x}|dx
 \nonumber=&\int_{\mathbb{R}}(\rho_{n}\ast w_{tx})sgn(\rho_{n,x}\ast w)dx
\\ \nonumber=&\int_{\mathbb{R}}\{\rho_{n}\ast [w_x(4u_{x}+4v_{x}-u_{xx}-v_{xx})]\}sgn(\rho_{n,x}\ast w)dx
\\ \nonumber&+\int_{\mathbb{R}}\{\rho_{n}\ast [(4v-u_{x}-v_{x})w_{xx})]\}sgn(\rho_{n,x}\ast w)dx
\\\nonumber&+\int_{\mathbb{R}}[\rho_{n}\ast (4u_{xx}w)]sgn(\rho_{n,x}\ast w)dx
\\ \nonumber&+\int_{\mathbb{R}}(\rho_{n}\ast G_{x}\ast[(u_{x}+v_{x})w_{x}])sgn(\rho_{n,x}\ast w)dx
\\\nonumber&+\int_{\mathbb{R}}(\rho_{n}\ast G_{xx}\ast[2(u_{x}+v_{x})w_{x}+6(u+v)w])sgn(\rho_{n,x}\ast w)dx
\\=&\sum_{i=1}^{5}I_{i},
\end{align}
where
\begin{align}
 \nonumber& I_{1}=\int_{\mathbb{R}}\bigg\{\rho_{n}\ast [w_x(4u_{x}+4v_{x}-u_{xx}-v_{xx})]\bigg\}sgn(\rho_{n,x}\ast w)dx,
\\ \nonumber& I_{2}=-\int_{\mathbb{R}}\bigg\{\rho_{n}\ast [(u_{x}+v_{x}-4v)w_{xx})]\bigg\}sgn(\rho_{n,x}\ast w)dx,
\\ \nonumber& I_{3}=\int_{\mathbb{R}}[\rho_{n}\ast (4u_{xx}w)]sgn(\rho_{n,x}\ast w)dx,
\\ \nonumber& I_{4}=\int_{\mathbb{R}}\bigg\{\rho_{n}\ast G_{x}\ast[(u_{x}+v_{x})w_{x}]\bigg\}sgn(\rho_{n,x}\ast w)dx,
\\ \nonumber& I_{5}=\int_{\mathbb{R}}\bigg\{\rho_{n}\ast G_{xx}\ast[2(u_{x}+v_{x})w_{x}+6(u+v)w]\bigg\}sgn(\rho_{n,x}\ast w)dx.
\end{align}
Due to (\ref{l13})-(\ref{l14}), we find for all $n\geq1$ and $t\in \mathbb{R}_{+}$
\begin{align}\label{l15}
\nonumber I_{1}\leq&\int_{\mathbb{R}}|\rho_{n}\ast [w_x(4u_{x}+4v_{x}-u-v)]|dx
+\int_{\mathbb{R}}\{\rho_{n}\ast [w_x(u-u_{xx}+v-v_{xx})]\}sgn(\rho_{n,x}\ast w)dx
\\\nonumber\leq&\|4u_{x}+4v_{x}-u-v\|_{L^\infty}\int_{\mathbb{R}}\rho_{n}\ast |w_x|dx
+\int_{\mathbb{R}}|\rho_{n}\ast w_x|[\rho_{n}\ast (u-u_{xx}+v-v_{xx})]dx
\\\nonumber &+\int_{\mathbb{R}}\bigg\{\rho_{n}\ast [w_x(u-u_{xx}+v-v_{xx})]-
\rho_{n}\ast(w_x)[\rho_{n}\ast(u-u_{xx}+v-v_{xx})]\bigg\}sgn(\rho_{n,x}\ast w)dx
\\\nonumber\leq&\|4u_{x}+4v_{x}-u-v\|_{L^\infty}\int_{\mathbb{R}}\rho_{n}\ast |w_x|dx
 \\\nonumber &-\int_{\mathbb{R}}|\rho_{n}\ast w_x|[\rho_{n}\ast (u_{xx}+v_{xx})]dx
+\|\rho_{n}\ast (u+v)\|_{L^\infty}\int_{\mathbb{R}}|\rho_{n}\ast w_x|dx
\\\nonumber &+\int_{\mathbb{R}}\bigg|\rho_{n}\ast [w_x(u-u_{xx}+v-v_{xx})]-
\rho_{n}\ast(w_x)[\rho_{n}\ast(u-u_{xx}+v-v_{xx})]\bigg|dx
\\\nonumber\leq&6M\int_{\mathbb{R}}|\rho_{n}\ast w_x|dx+5M\int_{\mathbb{R}}[\rho_{n}\ast |w_x|-|\rho_{n}\ast w_x|]dx-\int_{\mathbb{R}}|\rho_{n}\ast w_x|[\rho_{n}\ast (u_{xx}+v_{xx})]dx\\\nonumber &+\int_{\mathbb{R}}\bigg|\rho_{n}\ast [w_x(u-u_{xx}+v-v_{xx})]-
\rho_{n}\ast(w_x)[\rho_{n}\ast(u-u_{xx}+v-v_{xx})]\bigg|dx
\\\leq&6M\int_{\mathbb{R}}|\rho_{n}\ast w_{x}|dx+Q_{n}(t)-\int_{\mathbb{R}}|\rho_{n}\ast w_x|[\rho_{n}\ast (u_{xx}+v_{xx})]dx,
\\\nonumber
I_{2}=&\int_{\mathbb{R}}\{\rho_{n}\ast [(u_{x}+v_{x}-4v)(w-w_{xx})]\}sgn(\rho_{n,x}\ast w)dx
\\\nonumber&-\int_{\mathbb{R}}\{\rho_{n}\ast [(u_{x}+v_{x}-4v)w]\}sgn(\rho_{n,x}\ast w)dx
\\\nonumber=&
\int_{\mathbb{R}}[\rho_{n}\ast(w- w_{xx})][\rho_{n}\ast(u_{x}+v_{x}-4v)]sgn(\rho_{n,x}\ast w)dx
\\\nonumber&+\int_{\mathbb{R}}\bigg\{\rho_{n}\ast [(u_{x}+v_{x}-4v)(w-w_{xx})]
\\\nonumber&-[\rho_{n}\ast (w-w_{xx}) ] [\rho_{n}\ast(u_{x}+v_{x}-4v)]\bigg\}sgn(\rho_{n,x}\ast w)dx
\\\nonumber&
-\int_{\mathbb{R}}\{\rho_{n}\ast [(u_{x}+v_{x}-4v)w]\}sgn(\rho_{n,x}\ast w)dx
\\\nonumber\leq&-\int_{\mathbb{R}}(\frac{\partial}{\partial x}|\rho_{n}\ast w_{x}|)[\rho_{n}\ast(u_{x}+v_{x}-4v)]dx
+\int_{\mathbb{R}}[\rho_{n}\ast w][\rho_{n}\ast(u_{x}+v_{x}-4v)]sgn(\rho_{n,x}\ast w)dx
\\\nonumber&+\int_{\mathbb{R}}\bigg|\rho_{n}\ast [(4v-u_{x}-v_{x})(w-w_{xx})]dx
-[\rho_{n}\ast(w- w_{xx}) ] [\rho_{n}\ast(u_{x}+v_{x}-4v)]\bigg|dx
\\\nonumber&+\int_{\mathbb{R}}|\rho_{n}\ast [(u_{x}+v_{x}-4v)w]|dx
\\\nonumber\leq&\int_{\mathbb{R}}|\rho_{n}\ast w_{x}|[\rho_{n}\ast(u_{xx}+v_{xx}-4v_{x})]dx
\\\nonumber&+\int_{\mathbb{R}}\bigg|\rho_{n}\ast [(u_{x}+v_{x}-4v)(w-w_{xx})]
-[\rho_{n}\ast(w- w_{xx}) ] [\rho_{n}\ast(u_{x}+v_{x}-4v)]\bigg|dx
\\\nonumber&+\int_{\mathbb{R}}[\rho_{n}\ast w][\rho_{n}\ast(u_{x}+v_{x}-4v)]sgn(\rho_{n,x}\ast w)dx
+\int_{\mathbb{R}}\rho_{n}\ast |(u_{x}+v_{x}-4v)w|dx
\\\nonumber\leq&\|\rho_{n}\ast(4v_{x})\|_{L^\infty}\int_{\mathbb{R}}|\rho_{n}\ast w_{x}|dx
+\int_{\mathbb{R}}|\rho_{n}\ast w_{x}|[\rho_{n}\ast(u_{xx}+v_{xx})]dx
\\\nonumber&+\int_{\mathbb{R}}\bigg|\rho_{n}\ast [(u_{x}+v_{x}-4v)(w-w_{xx})]dx
-[\rho_{n}\ast(w-w_{xx})] [\rho_{n}\ast(u_{x}+v_{x}-4v)]\bigg|dx
\\\nonumber&+\|[\rho_{n}\ast(u_{x}+v_{x}-4v)]\|_{L^\infty}\int_{\mathbb{R}}|\rho_{n}\ast w|dx
+\|(u_{x}+v_{x}-4v)\|_{L^\infty}\int_{\mathbb{R}}\rho_{n}\ast |w|dx
\\\nonumber\leq&2M\int_{\mathbb{R}}|\rho_{n}\ast w_{x}|dx+6M\int_{\mathbb{R}}|\rho_{n}\ast w|dx
+3M\int_{\mathbb{R}}\bigg(\rho_{n}\ast| w|-|\rho_{n}\ast w|\bigg)dx
\\\nonumber&+\int_{\mathbb{R}}\bigg|\rho_{n}\ast [(u_{x}+v_{x}-4v)(w-w_{xx})]dx
-[\rho_{n}\ast(w-w_{xx})] [\rho_{n}\ast(u_{x}+v_{x}-4v)]\bigg|dx
\\\nonumber&+\int_{\mathbb{R}}|\rho_{n}\ast w_{x}|[\rho_{n}\ast(u_{xx}+v_{xx})]dx
\\ \leq&2M\int_{\mathbb{R}}|\rho_{n}\ast w_{x}|dx+6M\int_{\mathbb{R}}|\rho_{n}\ast w|dx
+\int_{\mathbb{R}}|\rho_{n}\ast w_{x}|[\rho_{n}\ast(u_{xx}+v_{xx})]dx+Q_{n}(t),
\\\nonumber
I_{3}\leq&\int_{\mathbb{R}}\rho_{n}\ast |4u_{xx}w|dx
\leq4\|u_{xx}\|_{L^\infty}\int_{\mathbb{R}}\rho_{n}\ast |w|dx
\\\nonumber\leq&10M\int_{\mathbb{R}}|\rho_{n}\ast w|dx
+10M\int_{\mathbb{R}}\bigg(\rho_{n}\ast |w|-|\rho_{n}\ast w|\bigg)dx
\\\leq&10M\int_{\mathbb{R}}|\rho_{n}\ast w|dx+Q_{n}(t),
\end{align}
with $Q_{n}(t)$ in the class (\ref{l21}).
By the same token with (\ref{l11}), we get
\begin{align}
\nonumber I_{4}
\leq&\|G_{x}\|_{L^1}\int_{\mathbb{R}}|\rho_{n}\ast [(u_{x}+v_{x})w_{x}]|dx
\\\nonumber\leq&\int_{\mathbb{R}}[\rho_{n}\ast |(u_{x}+v_{x})w_{x}||dx
\\\nonumber\leq&\|u_{x}+v_{x}\|_{L^\infty}\int_{\mathbb{R}}\rho_{n}\ast|w_{x}|dx
\\\nonumber\leq&M\int_{\mathbb{R}}|\rho_{n}\ast w_{x}|dx
+M\int_{\mathbb{R}}\bigg(\rho_{n}\ast |w_{x}|
-|\rho_{n}\ast w_{x}|\bigg)dx
\\\leq&M\int_{\mathbb{R}}|\rho_{n}\ast w_{x}|dx+Q_{n}(t),
\end{align}
with $Q_{n}(t)$ in the class (\ref{l21}).

Noting that $G_{xx}\ast f=G\ast f-f$, by the same toke with (\ref{l11}), we easily infer
\begin{align}\label{l16}
\nonumber I_{5}\leq&\int_{\mathbb{R}}|\rho_{n}\ast G_{xx}\ast[2(u_{x}+v_{x})w_{x}+6(u+v)w]|dx
\\\nonumber\leq&\int_{\mathbb{R}}|\rho_{n}\ast G\ast[2(u_{x}+v_{x})w_{x}+6(u+v)w]|dx
\\\nonumber&+\int_{\mathbb{R}}|\rho_{n}\ast [2(u_{x}+v_{x})w_{x}+6(u+v)w]|dx
\\\nonumber\leq&\|G\|_{L^1}\int_{\mathbb{R}}|\rho_{n}\ast[2(u_{x}+v_{x})w_{x}+6(u+v)w]|dx
\\\nonumber&+\int_{\mathbb{R}}|\rho_{n}\ast [2(u_{x}+v_{x})w_{x}+6(u+v)w]|dx
\\\nonumber\leq&12\int_{\mathbb{R}}\rho_{n}\ast|(u+v)w|dx
+4\int_{\mathbb{R}}\rho_{n}\ast|(u_{x}+v_{x})w_{x}|dx
\\\nonumber\leq&12\|u+v\|_{L^\infty} \int_{\mathbb{R}}\rho_{n}\ast |w|dx
+4\|u_{x}+v_{x}\|_{L^\infty}\int_{\mathbb{R}}\rho_{n}\ast |w_{x}|dx
\\\nonumber\leq&12M\int_{\mathbb{R}}|\rho_{n}\ast w|dx
+4M\int_{\mathbb{R}}|\rho_{n}\ast w_{x}|dx
\\\nonumber&+12M\int_{\mathbb{R}}\bigg(\rho_{n}\ast |w|-|\rho_{n}\ast w|\bigg)dx
+4M\int_{\mathbb{R}}\bigg(\rho_{n}\ast| w_{x}|-|\rho_{n}\ast w_{x}|\bigg)dx
\\\leq&4M\int_{\mathbb{R}}|\rho_{n}\ast w_{x}|dx+12M\int_{\mathbb{R}}|\rho_{n}\ast w|dx+Q_{n}(t),
\end{align}
with $Q_{n}(t)$ in the class (\ref{l21}).

Plugging (\ref{l15})-(\ref{l16}) into (\ref{l17}), we have
\begin{align}\label{l22}
\frac{d}{dt}&\int_{\mathbb{R}}|\rho_{n}\ast w_{x}|dx\leq 15\int_{\mathbb{R}}|\rho_{n}\ast w_x|dx+28\int_{\mathbb{R}}|\rho_{n}\ast w|dx+Q_{n}(t),
\end{align}
with $Q_{n}(t)$ in the class (\ref{l21}), which along with (\ref{l23}). And applying to Gronwall's inequality yields for all $t\in\mathbb{R}_{+}, ~\text{and}~~n\geq1,$
\begin{align}\label{l25}
\int_{\mathbb{R}}(|\rho_{n}\ast w_{x}|+|\rho_{n}\ast w|)(t,x)dx\leq \int_{0}^{t}e^{36M(t-s)}Q_{n}(s)ds
\\ \nonumber +[\int_{\mathbb{R}}(|\rho_{n}\ast w_{x}|+|\rho_{n}\ast w|)(t,x)dx]e^{36M(t)}.
\end{align}

Fix $t>0$, and let $n\rightarrow\infty$ in (\ref{l25}). Noticing that $Q_{n}$ satisfies (\ref{l21}), and $w,w_{x}\in L^p(\mathbb{R})~ with~ p\in[1,\infty).$ An application of Lebesgue's dominated convergence theorem leads to
\begin{align}\label{l26}
\int_{\mathbb{R}}(|\rho_{n}\ast w_{x}|+|\rho_{n}\ast w|)(t,x)dx\leq [\int_{\mathbb{R}}(|\rho_{n}\ast w_{x}|+|\rho_{n}\ast w|)(t,x)dx]e^{36M(t)},~~t\in\mathbb{R}.
\end{align}

This complete the proof of uniqueness.

\end{proof}

\noindent\textbf{Acknowledgements}. This work was
partially supported by NNSFC (No.11271382), RFDP (No.
20120171110014), MSTDF (No. 098/2013/A3), and Guangdong Special Support Program (No. 8-2015).

\phantomsection
\addcontentsline{toc}{section}{\refname}


\begin{thebibliography}{99}
\small

\bibitem{C-K} G. M. Coclite and K. H. Karlsen, \textit{On the well-posedness of the Degasperis-Procesi equation}, { J. Func. Anal.}, {\bf 233} (2006), 60--91.

\bibitem{Bressan.Constantin}A. Bressan and A. Constantin, \textit{Global conservative solutions of the Camassa-Holm equation}, {Arch. Ration. Mech. Anal.}, {\bf183} (2007), 215--239.

\bibitem{Bressan.Constantin2} A. Bressan and A. Constantin, \textit{Global dissipative solutions of the Camassa-Holm equation}, {Anal. Appl. (Singap.)}, {\bf5} (2007), 1--27.

\bibitem{Camassa}R. Camassa and D. D. Holm, \textit{An integrable shallow water equation with peaked solitons}, {Phys. Rev. Lett.}, {\bf71}  (1993), 1661--1664.

\bibitem{Camassa.Hyman} R. Camassa, D. Holm and J. Hyman, \textit{A new integrable shallow water
equation}, {Adv. Appl. Mech.}, {\bf31} (1994), 1--33.

\bibitem{Constantin-E} A. Constantin, \textit{The Hamiltonian structure of the Camassa-Holm equation}, {Exposition. Math.}, {\bf 15(1)} (1997), 53--85.

\bibitem{Constantin-P} A. Constantin, \textit{On the scattering problem for the Camassa-Holm equation}, {R. Soc. Lond. Proc. Ser. A}, {\bf 457} (2001), 953--970.

\bibitem{Constantin} A. Constantin, \textit{Existence of permanent and breaking waves for a shallow water equation: a geometric approach}, {Ann. Inst. Fourier (Grenoble)}, {\bf50} (2000), 321--362.

\bibitem{Constantin2} A. Constantin, \textit{The trajectories of particles in Stokes waves}, {Invent. Math.}, {\bf166} (2006), 523--535.

\bibitem{Constantin.Escher} A. Constantin and J. Escher, \textit{Global existence and blow-up for a shallow water equation}, {Ann. Scuola Norm. Sup. Pisa Cl. Sci.},  {\bf26} (1998), 303--328.

\bibitem{Constantin.Escher2} A. Constantin and J. Escher, \textit{Well-posedness, global existence, and
blowup phenomena for a periodic quasi-linear hyperbolic equation}, { Comm. Pure Appl. Math.}, {\bf51} (1998), 475--504.

\bibitem{Constantin.Escher3} A. Constantin and J. Escher, \textit{Wave breaking for nonlinear nonlocal shallow water equations}, {Acta Math.}, {\bf181} (1998), 229--243.

\bibitem{Constantin.Escher4} A. Constantin and J. Escher, \textit{Particle trajectories in solitary water waves}, {Bull. Amer. Math. Soc.}, {\bf44} (2007), 423--431.

\bibitem{Constantin.Escher5} A. Constantin and J. Escher, \textit{Analyticity of periodic traveling free surface water waves with vorticity}, {Ann. of Math.}, {\bf173} (2011), 559--568.

\bibitem{Constantin.Lannes} A. Constantin and D. Lannes, \textit{The hydrodynamical relevance of the
Camassa-Holm and Degasperis-Procesi equations}, {Arch. Ration. Mech. Anal.}, {\bf192} (2009), 165--186.

\bibitem{Constantin.mckean} A. Constantin and H. P. McKean, \textit{A shallow water equation on the circle}, {Comm. Pure Appl. Math.}, {\bf55} (1999), 949--982.


\bibitem{Constantin.Molinet} A. Constantin and L. Molinet, \textit{Global weak solutions for a shallow water equation},  { Comm. Math. Phys.}, {\bf211} (2000), 45--61.

\bibitem{Constantin.lvanov.lenells} A. Constantin, R. I. Ivanov and J. Lenells, \textit{Inverse scatterering transform for the Degasperis-Procesi equation}, {Nonlinearity}, {\bf23} (2010), 2559--2575.

\bibitem{Constantin.Strauss}  A. Constantin and W. A. Strauss, \textit{Stability of peakons}, {Comm. Pure Appl. Math.}, {\bf53} (2000), 603--610.

\bibitem{d1}R. Danchin, \textit{A few remarks on the Camassa-Holm equation}, {Differential Integral Equations}, {\bf14}  (2001), 953--988.

\bibitem{d2}R. Danchin, \textit{A note on well-posedness for Camassa-Holm equation}, {J. Differential Equations} {\bf192} (2003), 429-444.


\bibitem{D-H-H} A. Degasperis, D. D. Holm, and A. N. W. Hone, {\it A new integral equation with peakon solutions}, {Theor. Math. Phys.}, {\bf 133} (2002), 1463--1474.

\bibitem{D-P} A. Degasperis, M. Procesi, \textit{ Asymptotic integrabilityAsymptotic integrability // Symmetry and Perturbation Theory}, {World Sci. Publ., River Edge, NJ}, {\bf1(1)} (1999), 23-37.

\bibitem{D-G-H} H. R. Dullin, G. A. Gottwald, and D. D. Holm, {\it On asymptotically equivalent shallow water wave equations}, {Phys. D}, {\bf 190} (2004), 1--14.

\bibitem{lions1} P. L. Lions, \textit{Mathematical Topics in Fluid Mechanics, vol. I. Incompressible Models}, {Oxford Lecture Ser. Math. Appl.}, {\bf vol. 3}, Clarendon, Oxford University Press, New York, 1996.


\bibitem{E-L-Y1}
J. Escher, Y. Liu and Z. Yin, {\it Global weak solutions and blow-up structure for the Degasperis-Procesi  equation}, {J. Funct. Anal.}, {\bf 241} (2006), 457--485.

\bibitem{E-L-Y}
J. Escher, Y. Liu and Z. Yin, {\it Shock waves and blow-up phenomena for the periodic Degasperis-Procesi  equation}, {Indiana Univ. Math. J.}, {\bf 56} (2007), 87--177.


\bibitem{Fokas} A. Fokas and B. Fuchssteiner, \textit{Symplectic structures, their B\"{a}cklund transformation and hereditary symmetries}, {Phys. D}, {\bf 4(1)} (1981/82), 47--66.

\bibitem{gy1}C. Guan and Z. Yin,\textit{Global weak solutions for a two-component Camassa-Holm shallow water system}, {J. Funct. Anal.}, {\bf260} (2011), 1132¨C1154.

\bibitem{G-L} G. Gui and Y. Liu, \textit{On the Cauchy problem for the Degasperis-Procesi equation}, {Quart. Appl. Math.}, {\bf 69}, 445-464, (2011).

\bibitem{H-H} A. A. Himonas and C. Holliman, \textit{The Cauchy problem for the Novikov equation}, {Nonlinearity}, {\bf 25} (2012), 449-479.

\bibitem{Hone}A. N. W. Hone and J. Wang, \textit{Integrable peakon equations with cubic nonlinearity}, {J. Phys. A }, {\bf41} (2008), 372002, 10pp.


\bibitem{Laishaoyong}S. Lai, \textit{Global weak solutions to the Novikov equation}, {J. Funct. Anal.}, {\bf265} (2013), 520-544.

\bibitem{Le}
J. Lenells, \textit{ Traveling wave solutions of the Degasperis-Procesi
equation}, {J. Math. Anal. Appl.} {\bf 306} (2005), 72--82.

\bibitem{L-Y1}
Y. Liu and Z. Yin, \textit{ Global Existence and Blow-up Phenomena for the
Degasperis-Procesi Equation}, {Commun. Math. Phys.}, {\bf 267}
(2006), 801--820.

\bibitem{L-Y2}
 Y. Liu and Z. Yin,  \textit{On the blow-up phenomena for the Degasperis-Procesi equation}, {Int. Math. Res. Not. IMRN}, {\bf 23} (2007), rnm117, 22 pp.

\bibitem{Lu}
H. Lundmark, \textit{ Formation and dynamics of shock waves in the Degasperis-Procesi equation}, {J. Nonlinear. Sci.}, {\bf 17} (2007), 169--198.

\bibitem{Lu-S} H. Lundmark and J. Szmigielski, \textit{ Multi-peakon solutions of the Degasperis-Procesi equation}, {Inverse Prob.}, {\bf 19} (2003), 1241-1245.

\bibitem{liu-yin1} Y. Liu and Z. Yin, \textit{ Global existence and blow-up phenomena
for the Degasperis-Procesi equation}, {Commun. Math. Phys.}, {\bf267} (2006), 801¨C820.



\bibitem{malek1} J. Malek,  J. Necas,  M. Rokyta and  M. Ruzicka,\textit{ Weak and Measure-valued Solutions to Evolutionary PDEs}, {London: Chapman \& Hall}, (1996), xii+317 pp..


\bibitem{Natanson} I. P. Natanson,\textit{ Theory of Functions of a Real Variable}, {NewYork: Frederick Ungar Publishing Co.}, {\bf26} (1961), 265pp.




\bibitem{n1} V. Novikov, \textit{Generalization of the Camassa-Holm equation}, {J. Phys. A}, {\bf 42} (2009), 342002, 14pp.

\bibitem{Guillermo} G. Rodr\'{i}guez-Blanco, \textit{On the Cauchy problem for the Camassa-Holm equation}, {Nonlinear Anal. Ser. A}, {\bf46} (2001), 309-327.

\bibitem{Toland}  J. F. Toland,  \textit{Stokes waves}, {Topol. Methods Nonlinear Anal.}, {\bf7} (1996), 1--48.

\bibitem{Tu-Yin1} X. Tu and Z. Yin,  \textit{Local well-posedness and blow-up phenomena for a generalized Camassa-Holm equation with peakon solutions}, { Discrete Contin. Dyn. Syst. A}, {\bf 128} (2016), 1--19

\bibitem{Tu-Yin2} X. Tu and Z. Yin,  \textit{Blow-up phenomena and local well-posedness for a generalized Camassa-Holm equation in the critical Besov space}, {Nonlinear Anal. TMA}, {\bf 128} (2015), 1--19.

\bibitem{V-P} V. O. Vakhnenko and E. J. Parkes,  \textit{Periodic and solitary-wave solutions of the Degasperis-Procesi equation}, {Chaos Solitons Fractals}, {\bf 20} (2004), 1059--1073.

\bibitem{Wu.Yin} X. Wu and Z. Yin, \textit{Global weak solutions for the Novikov equation}, {J. Phys. A}, {\bf44} (2011), 055202, 17pp.

\bibitem{Wu.Yin2}X. Wu and Z. Yin, \textit{Well-posedness and global existence for the Novikov equation}, {Annali della Scuola Normale Superiore di Pisa. Classe di Scienze. Serie V}, {\bf11} (2012), 707--727.

\bibitem{Wu.Yin3}X. Wu and Z. Yin, \textit{A note on the Cauchy problem of the Novikov equation}, { Appl. Anal.}, {\bf92} (2013), 1116--1137.

\bibitem{Xin.Z.P}Z. Xin and P. Zhang, \textit{On the weak solutions to a shallow water equation}, { Comm. Pure Appl. Math.}, {\bf53} (2000), 1411--1433.

\bibitem{Wei.Yan}W. Yan, Y. Li and Y. Zhang, \textit{The Cauchy problem for the integrable Novikov equation}, {J. Differential Equations}, {\bf253} (2012), 298--318.

\bibitem{Wei.Yan2}W. Yan, Y. Li and Y. Zhang, \textit{The Cauchy problem for the Novikov equation}, {NoDEA Nonlinear Differential Equations Appl.}, {\bf20} (2013), 1157--1169.

\bibitem{y1}Z.  Yin, \textit{On the Cauchy problem for an integrable equation with peakon solutions}, {Ill. J. Math.}, {\bf 47} (2003), 649--666.

\bibitem{y2} Z. Yin, \textit{ Global existence for a new periodic integrable equation}, {J. Math. Anal. Appl.}, {\bf 283} (2003), 129--139.

\bibitem{y3} Z. Yin, \textit{ Global weak solutions to a new periodic integrable equation with peakon solutions}, {J. Funct. Anal.}, {\bf 212} (2004), 182--194.

\bibitem{y4} Z. Yin, \textit{ Global solutions to a new integrable equation with peakons}, {Indiana Univ. Math. J.}, {\bf 53} (2004), 1189--1210.
\end{thebibliography}
\end{document}